\documentclass[12,reqno]{amsart}
 \usepackage{graphicx,amssymb,amscd,amsxtra,amsmath,amsfonts,amsthm}
 \usepackage[nobysame,alphabetic]{amsrefs}
 \usepackage{float}
 \usepackage{tikz}
 \usetikzlibrary{arrows}
 \usepackage{enumerate}

\newcounter{num}

\theoremstyle{plain}
\newtheorem{theorem}[num]{Theorem}
\newtheorem{lemma}[num]{Lemma}
\newtheorem{cor}[num]{Corollary}
\newtheorem{prop}[num]{Proposition}

\newtheorem*{thm*}{Theorem}
\newtheorem*{lemma*}{Lemma}
\newtheorem*{prop*}{Proposition}
\newtheorem*{cor*}{Corollary}

\theoremstyle{remark}

\newcommand{\N}{{\mathbb N}}

\newcommand{\R}{{\mathbb R}}
\newcommand{\BS}{{S}}

\newcommand{\Z}{{\mathbb Z}}

\newcommand{\vK}{{\mathcal K}}
\newcommand{\vL}{{\mathcal L}}

\newcommand{\vT}{{\mathcal T}}

\newcommand{\dom}{\preccurlyeq}

\newcommand{\vp}{\varphi}
\newcommand{\e}{\varepsilon}
\newcommand{\MCG}{\mbox{MCG}}
\newcommand{\vPML}{\mathcal{PML}}

\newcommand{\gp}[3]{(#2 \cdot #3)_{#1}}
\renewcommand{\P}{\mathbb{P}}

\renewcommand{\le}{\leqslant}
\renewcommand{\ge}{\geqslant}

\newcommand{\rmu}{\check \mu}

\newcommand{\norm}[1]{\left| {#1} \right|}

\newcommand{\dsp}{d_{\text{sp}}}

\newcommand{\supp}{\text{supp}}

\newcounter{case}

\newenvironment{case}{\stepcounter{case} \addvspace{.5\baselineskip} \noindent\textbf{Case \thecase}.}{\hfill\fbox{Case \thecase}}

\DeclareMathSymbol
{\rightrightarrows}
{\mathrel}{AMSa}{"13}

\begin{document}

\author{Alexander Lubotzky, Joseph Maher and Conan Wu}

\title{Random methods in 3-manifold theory}

\maketitle

{\hfill \today \hfill}

\tableofcontents

\section{Introduction}

Over the years, random methods have evolved into powerful techniques
in several areas of mathematics. Most notably, as pioneered by Paul
Erd\H{o}s \cite{erdos}, the study of random graphs has become an
important branch of contemporary graph theory. What is especially
fascinating about this development is the fact that such techniques
solved many problems which have nothing to do with probability: one
can use random constructions to show the existence of a graph
satisfying particular properties without constructing explicit
examples.  For example, graphs with both arbitrarily large girth and
arbitrarily large chromatic number were shown to exist by random
methods by Erd\H{o}s \cite{erdos} long before explicit examples were
found by Lov\'asz \cite{lovasz} and Lubotzky, Phillips and Sarnak
\cite{lps}.

The goal of this paper is to present similar methods within the world
of 3-manifolds. In 2006, Dunfield and W. Thurston \cite{DT} presented
a model of `random' 3-manifolds by considering random walks on the
mapping class group, and a theory of random 3-manifolds is starting to
emerge (\cite{dt2}, \cite{R}, \cite{K}, \cite{M1}, \cite{dw},
\cite{M3}). Here we will use this theory to prove the following
existence result, which a priori has nothing to do with randomness.

\begin{theorem} \label{theorem:main} %
For any integers $k$ and $g$ with $g \geq 2$, there exist infinitely
many closed hyperbolic $3$-manifolds which are integral homology
$3$-spheres with Casson invariant $k$ and Heegaard genus equal to $g$.
\end{theorem}

In fact, results announced by Brock and Souto \cite{brock-souto} would
show that the volume of these $3$-manifolds tends to infinity.

There is however a difference between our methodology and the common
practice of random methods in graph theory. In graph theory, usually
one proves ``0-1 laws'' and existence is shown by proving that
``most'' objects satisfy the desired property, even though no explicit
examples are given. Here our considerations will be somewhat more
delicate: we will have to compare rates of decay of various properties
along random walks, and the difference between them will ensure the
existence of the desired manifolds.  We hope that our results will be
an initial example of applying random methods to $3$-manifolds.

\subsection{Outline}

In this section we give a brief outline of the proof.  By a classical
result (see, for example \cite{H}), every closed 3-manifold can be
obtained by gluing two genus $g$ handlebodies along their boundary
surfaces $S_g$, and the minimum such $g$ is called the Heegaard genus
of the manifold.  As isotopic gluing maps give homeomorphic
$3$-manifold, every element $\phi$ of the mapping class group
$\MCG(S_g)$ will give a $3$-manifold $M(\phi)$ of Heegaard genus at
most $g$. The main idea of \cite{DT} is that a random walk on
$\MCG(S_g)$ gives rise to a random model of $3$-manifolds with
Heegaard genus at most $g$; one may also consider random walks on any
subgroup $H$ of $\MCG(S_g)$.

If $w_n$ is a random walk on a group $H$ generated by the probability
distribution $\mu$, and $Y$ is a subset of $H$, we say that $Y$ is
\emph{exponentially small with respect to $\mu$} if the probability of
visiting $Y$ decays exponentially fast with $n$. We say the set $Y$ is
\emph{exponentially large with respect to $\mu$} if the complement of
$Y$ is exponentially small.  We will often just write
\emph{exponentially small} or \emph{exponentially large} if the
probability distribution $\mu$ is clear from context. We do not
necessarily require that $\mu$ be symmetric, however, we always
require that the semi-group generated by the support of $\mu$ is a
group. The main idea of the proof is to find a specifically chosen
finitely generated subgroup $H$ of the Torelli subgroup $\mathcal{T}$
of the $\MCG(S_g)$, and for random walks on this subgroup we show:

\begin{enumerate}[(a)]

\item The set of elements of $H$ giving rise to hyperbolic manifolds
is exponentially large.

\item The set elements of $H$ giving rise to manifolds of Heegaard
genus exactly $g$ is exponentially large.

\item The Casson invariant restricted to $H$ is a homomorphism from
$H$ onto $\Z$.

\end{enumerate}

Theorem \ref{theorem:main} follows from (a), (b) and (c). Indeed by
standard results from random walks on $\Z$, the random walk $w_n$
visits each $k \in \Z$ with probability $1/\sqrt{n}$, for $n$
sufficiently large.  By (a) and (b) most of these visits will give
rise to hyperbolic manifolds of Heegaard genus exactly $g$. The
resulting manifolds will also be integral homology spheres, as we
shall choose the subgroup $H$ to be contained in the Torelli subgroup.

In more detail, we choose $H$ to be a subgroup of $\vK$, the group
generated by Dehn twists in separating curves. By a result of Morita
\cite{Mo}, the Casson invariant $\lambda \colon \vK \to \Z$ is an
epimorphism. While $\vK$ is not expected to be finitely generated, we
choose a sufficiently large finitely generated subgroup $H$ of $\vK$
for which $\lambda$ restricted to $H$ still has surjective image in
$\Z$.

For this subgroup $H$ we prove properties (a) and (b). Moreover, much
of what we prove holds for complete subgroups of the mapping class
group, i.e. subgroups whose limit set is equal to the full
boundary. To put this in perspective, let us mention that Maher
\cite{M1} showed that in complete, finitely generated subgroups, the
probability that a random walk gives rise to a manifold which is
hyperbolic, and of Heegaard genus $g$, tends to $1$, but without the
exponential decay estimate. We now have the following result which is
of independent interest. We shall write $\supp(\mu)$ for the support
of $\mu$, i.e. all group elements $g \in G$ with $\mu(g) > 0$. We
shall write $\langle \supp(\mu) \rangle_+$ for the semi-group
generated by $\supp(\mu)$, and we shall refer to this as the
semi-group support of $\mu$.

\begin{theorem} \label{theorem:2} %
Let $L$ be a finitely generated complete subgroup of $\MCG(S_g)$, then
for any finitely supported probability distribution $\mu$, whose
semi-group support $\langle \supp(\mu) \rangle_+$ is equal to $L$, the
set of elements which yield hyperbolic manifolds of Heegaard genus
equal to $g$ is exponentially large.
\end{theorem}

This result is new even for $L = \MCG(S_g)$ or the Torelli group. In
fact, we prove a more general result, which includes the subgroup $H$
as before, which need not be complete.

The real work in proving Theorem 2 is to control the Heegaard
splitting distance. It is known that if the Heegaard distance of $\phi
\in \MCG(S_g)$ is at least 3 then $M(\phi)$ is hyperbolic, by work of
Kobayashi \cite{kobayashi}, Hempel \cite{hempel} and Perelman
\cite{morgan-tian}, and if the spitting distance is at least $2g+1$
then $M(\phi)$ has Heegaard genus exactly $g$, by work of Scharlemann
and Tomova \cite{st}. In fact, we show that the Heegaard splitting
distance of the random $3$-manifold grows linearly with exponential
decay. Results announced by Brock and Souto \cite{brock-souto} would
then imply that the manifolds obtained in Theorem \ref{theorem:main}
have arbitrarily large volume.

\subsection{Acknowledgments}

The authors are grateful to Martin Bridson, Nathan Dunfield, Benson
Farb, Alexander Holroyd, Justin Malestein and Yair Minksy for useful
discussions.  We acknowledge support by ERC, NSF and ISF. The second
author was supported by PSC-CUNY award TRADB-45-17 and Simons
Foundation grant CGM 234477. The third author thanks GARE network and
the warm hospitality of Hebrew University.

\section{Proof of the main theorem}

Before starting the proof we give some background and fix notation.
Let $S_g$ be a closed orientable surface of genus $g$.  We shall write
$\MCG(S_g)$ for the mapping class group of $S_g$, which is the group
of all orientation preserving homeomorphisms of $S_g$ up to
isotopy. We shall write $\mathcal{C}(S_g)$ for the curve complex of
$S_g$, which is a simplicial complex, whose vertices are given by
isotopy classes of simple closed curves, and whose simplices are
spanned by collections of disjoint simple closed curves. A handlebody
$U$ is a compact $3$-manifold with boundary, homeomorphic to a regular
neighbourhood of an embedded graph in $\R^3$, and handlebodies are
classified up to homeomorphism by the genus of their boundary
surfaces.  Given an identification of the surface $S_g$ with the
boundary of a genus $g$ handlebody, the handlebody group is the
subgroup of the mapping class group consisting of those mapping class
group elements which extend over the handlebody, i.e. they arise as
restrictions of self-homeomorphisms of the handlebody. The disc set
$\mathcal{D}$ is defined to be the subset of the curve complex
$\mathcal{C}(S_g)$ consisting of all simple closed curves in $S_g$
which bound discs in the handlebody.  A genus $g$ Heegaard splitting
of a closed orientable $3$-manifold $M$ is an embedding of $S_g$ in
$M$ which divides $M$ into two handlebodies $U$ and $U'$, and we shall
denote their corresponding discs sets by $\mathcal{D}$ and
$\mathcal{D}'$ respectively. Any two handlebodies of the same genus
are homeomorphic, so for any pair of discs sets $\mathcal{D}$ and
$\mathcal{D}'$, corresponding to two identifications of $S_g$ with the
boundaries of the handlebodies, there is a mapping class group element
$h$ such that $\mathcal{D}' = h\mathcal{D}$. The mapping class group
element is not unique, but any two choices differ by composition with
elements of the handlebody group.

In particular, a Heegaard splitting of $\BS^3$ is an embedded copy of
the surface $S_g$ in the standard 3-sphere, separating the sphere into
two genus $g$ handlebodies. In fact, for $\BS^3$, such a splitting is
unique up to isotopy, and from now on we shall fix a pair of discs
sets $\mathcal{D}$ and $\mathcal{D'}$, and a mapping class group
element $h_{\BS^3}$ with $h_{\BS^3}\mathcal{D} = \mathcal{D'}$,
arising from a genus $g$ Heegaard splitting of $\BS^3$. Given an
element $\phi$ of the mapping class group, we may consider the
Heegaard splitting obtained by composing the gluing map $h_{\BS^3}$
with $\phi$, i.e. the Heegaard splitting with disc sets $\mathcal{D}$
and $\phi h_{\BS^3} \mathcal{D}$, and we shall just write $M(\phi)$
for the resulting $3$-manifold. The $3$-manifold $M(\phi)$ is an
integral homology sphere if and only if $\phi$ lies in the Torelli
subgroup $\vT$ of the mapping class group $\MCG(S_g)$, i.e. the
subgroup which acts trivially on the homology of the surface.

We will often consider the orbit map from the mapping class group
$\MCG(S_g)$ to the curve complex $\mathcal{C}(S)$, which sends $\phi
\mapsto \phi x_0$, for some choice of basepoint $x_0$. Our particular
choice of disc sets $\mathcal{D}$ and $\mathcal{D}'$ intersect,
i.e. in the unique genus $g$ Heegaard splitting of $S^3$ there is a
curve on the Heegaard surface which bounds a disc on both sides, and
it will be convenient for us to choose a basepoint $x_0$ which lies in
both $\mathcal{D}$ and $\mathcal{D}'$. We remark that this is for
convenience only, as the argument works for any other choice of
basepoint, possibly with slightly different constants. Furthermore,
the argument works for any other choice of initial disc sets,
$\mathcal{D}$ and $\mathcal{D}'$, again possibly with different
constants, as long as the Heegaard splitting corresponding to the two
disc sets is an integer homology sphere; this is equivalent to
starting the random walk at some other element of the mapping class
group, instead of the usual choice of the identity element.

Let $\mu$ be a probability distribution on $\MCG(S_g)$ with finite
support. The random walk on $\MCG(S_g)$ generated by $\mu$ is the
Markov chain with the transition probability from $x$ to $y$, denoted
$p(x, y)$, equal to $\mu(x^{-1}y)$.  We will always assume that we
start at the identity at time zero, and we will write $w_n$ for the
location of the random walk at time $n$.  The probability distribution
$\mu$ need not be symmetric, but we shall always assume that the
semi-group generated by the support of $\mu$ is a subgroup of the
mapping class group. Taking the probability of the random walk landing
in a particular set gives rise to a way of measuring the size of
subsets of our group: we say that a subset $E \subseteq G$ is
\emph{exponentially small} if there exists numbers $K, c < 1$ such
that for every $n \in \N$,
$$ \P(w_n \in E) \leq K c^n.$$
We will call a subset of $G$ \emph{exponentially large} if its
complement is exponentially small.

Given $g \in \MCG(S_g)$ which is pseudo-Anosov, let $\vL^s(g)$ and
$\vL^u(g)$ denote the stable and unstable laminations of $g$.  A
finitely generated subgroup $G$ of $\MCG(S_g)$ is said to be
\emph{sufficiently large} if it contains two pseudo-Anosov elements
$\vp, \ \psi$ with distinct stable and unstable laminations, namely
$$\{ \vL^s(\vp), \ \vL^u(\vp) \} \cap \{ \vL^s(\psi), \ 
\vL^u(\psi) \} = \phi.$$
The subgroup $G$ is \emph{complete} if the endpoints of its
pseudo-Anosov elements are dense in $\vPML$, Thurston's boundary for
Teichm\"uller space.  Ivanov \cite{ivanov} showed that an infinite
normal subgroup contains a pseudo-Anosov element, and the orbit under
$\MCG(S_g)$ of any point in $\vPML$ is dense, so every infinite normal
subgroup of $\MCG(S_g)$ is complete and sufficiently large.

Let $\vK$ be the subgroup of $\vT$ generated by Dehn twists along
separating curves. As $\vK$ is a normal subgroup, it is complete.  In
fact, as discussed in Farb and Margalit \cite{FM}*{Chapter 6}, the
group $\vK$ coincides with the second Torelli group, $\vT^2(S_g)$,
also known as the Johnson kernel, defined as the kernel of the action
of $\vT(S_g)$ on the quotient $\Gamma /[\Gamma, \Gamma']$ where
$\Gamma = \pi_1(S_g)$ and $\Gamma' = [\Gamma, \Gamma]$ is the
commutator subgroup.

The Casson invariant $\lambda$ of closed orientable integral homology
spheres takes values in $\Z$. Casson defined the invariant in terms of
$SU(2)$ representations arising from a Heegaard splitting of the
manifold, see for example Akbulut and McCarthy \cite{am}. The Casson
invariant of $\BS^3$ is equal to $0$ and the Casson invariant of the
Poincar\'e homology sphere is equal to $1$. As the sign of
$\lambda(M)$ changes if you reverse the orientation on $M$, and the
Casson invariant is additive under connect sum, taking connect sums of
Poincar\'e homology spheres with appropriate orientation gives
examples of manifolds with any integral value for their Casson
invariant, though these manifolds will not be hyperbolic, and their
Heegaard genera are unbounded.

We shall use the following property of the Casson invariant which is
due to Morita \cite{Mo}.

\begin{theorem}(Morita \cite{Mo})
The Casson invariant
$\lambda: \vK \rightarrow \Z$ is a homomorphism.
\end{theorem}

In fact, this homorphism is surjective, and we show this by
constructing an explicit element of the Johnson kernel $\vK$ which
maps to a generator of $\Z$.

\begin{lemma} \label{lemma:tau} %
The Casson invariant homomorphism $\lambda \colon \vK \rightarrow \Z$ is
surjective.
\end{lemma}

Given a knot $\kappa$ in $\BS^3$, we will write $\BS^3 + (p/q) \kappa$
for the $3$-manifold obtained by $(p/q)$-Dehn surgery along $\kappa$
in $\BS^3$.

\begin{proof}
By the surgery formula for the Casson invariant, see for example
\cite{Sa}*{Section 3.2.8}, for Dehn surgeries on the trefoil knot
$\kappa$ we have
$$| \lambda (\BS^3 + \frac{1}{m+1}\kappa) -
\lambda(\BS^3+\frac{1}{m}\kappa) |= 1, $$
independent of the integer $m$.

Now we can embed $\kappa$ as a separating curve $C$ in the closed
genus two surface as follows:

\begin{figure}[htbp]
\label{trefoil}
\centering
\includegraphics[scale=0.90]{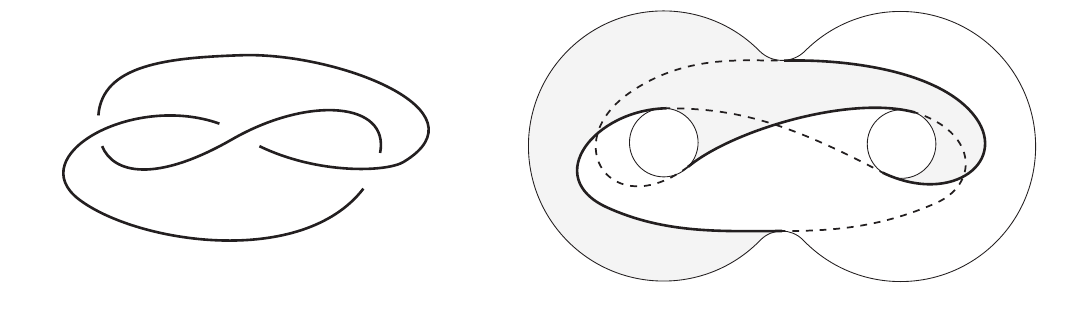}
\caption{Embedded separating trefoil}
\end{figure}

It is a standard fact (see, for example, \cite{Sa}*{Section 3.2.5})
that the $\BS^3 + \frac{1}{m}\kappa$ is the same as taking the
standard embedding of $S_2$ in $\BS^3$ and gluing the two handlebodies
with $m$ Dehn twists along the embedded copy of $\kappa$.

Let $\tau \in \vK$ be the Dehn twist along $C$, so we
have $$|\lambda(\tau^2) - \lambda(\tau)| = 1.$$ This finishes the
proof when $g=2$.

As the Casson invariant is additive under connected sums, hence for
any $g>2$ we may simply add handles on the standard embedding of $S_2$
away from the embedding of $\kappa$ and obtain the same manifold by
applying $m$ Dehn twists along $\kappa$.

Since $\vK$ is the group generated by Dehn twists along separating
curves in $S_g$, we deduce that for all $m$, $\BS^3+\frac{1}{m}\kappa$
are manifolds obtained as $M(\phi)$ for some $\phi \in \vK$. We
conclude there are always consecutive integers in the image
$\lambda(\vK)$, i.e. $\lambda|_{\vK}$ is surjective.
\end{proof}

Given subsets $A$ and $B$ of the curve complex, we will write $d(A,
B)$ for the minimum distance $d(a, b)$ between elements $a \in A$ and
$b \in B$.  The Heegaard splitting distance of a Heegaard splitting,
as defined by Kobayashi \cite{kobayashi2} and Hempel \cite{hempel}, is
the distance between the two disc sets determined by the handlebodies
of the splittings. In our notation, given a mapping class group
element $\phi$, the corresponding $3$-manifold $M(\phi)$ has Heegaard
splitting distance
\[ \dsp(M(\phi)) = d(\mathcal{D}, \phi h_{S^3} \mathcal{D} ). \]

In \cite{M1}, it was shown that for random walks on finitely generated
complete subgroups of the mapping class group, the splitting distance
$\dsp(M(w_n))$ grows linearly with $n$, i.e. there is a number $L > 0$
such that $P( \dsp(M(w_n)) \ge Ln ) \to 1$ as $n \to \infty$.
Although $\vK$ is complete, it is not expected to be finitely
generated, so we need a stronger version of this result which works
for subgroups, which need not be complete, and which furthermore shows
that the probability tends to $1$ exponentially fast. We say a
sequence of random variables $\{ X_n \}_{n \in \N}$ \emph{grow
  linearly with exponential decay} if there are numbers $K, L > 0$ and
$c < 1$ such that
\[  \P( X_n \le Ln ) \le K c^n, \]
for all $n$.  

\begin{theorem} \label{Hg} %
For any complete subgroup $G$ of $\MCG(S_g)$, there is a finitely
generated subgroup $H < G$, such that for any finitely supported
probability distribution $\mu$, whose semi-group support $\langle
\supp(\mu) \rangle_+$ is a subgroup containing $H$, the Heegaard
splitting distance $\dsp(M(w_n))$, of a random walk of length $n$
generated by $\mu$, grows linearly with exponential decay, i.e. there
are numbers $K, L > 0$ and $c < 1$ such that
\[ \P( \dsp(M(w_n)) \le Ln ) \le K c^n. \]
\end{theorem}

Any sufficiently large normal subgroup of the mapping class group is
complete, so in the result above $G$ may be taken to be the entire
$\MCG(S_g)$, the Torelli group, or the Johnson kernel $\vK$.  If $G$
is finitely generated, we may choose the support of $\mu$ to generate
$G$.  The mapping class group $\MCG(S_g)$, is finitely generated, as
shown by Dehn \cite{dehn} and Lickorish \cite{lickorish}, as is the
Torelli group, for $g > 2$, as shown by Johnson \cite{johnson}. For
$g=2$, the Torelli group is not finitely generated, as shown by
McCullough and Miller \cite{mcc-miller}, and it is not currently known
whether or not $\vK$ is finitely generated.

We postpone the proofs of Theorem \ref{Hg} to the later sections.

Kobayashi \cite{kobayashi} and Hempel \cite{hempel} showed that if the
splitting distance is greater than $2$ then the $3$-manifold is
irreducible, atoroidal and not Seifert-fibered, and so is hyperbolic
by Perelman's proof of Thurston's geometrization conjecture
\cite{morgan-tian}. Scharleman and Tomova \cite{st} showed that if the
Heegaard splitting distance $\dsp(M(h))$ is greater than $2g$, then the
Heegaard genus of the resulting $3$-manifold is equal to $g$. As the
set $H_{2g}$ of elements in $H$ that induce $3$-manifolds of splitting
distance at most $2g$ is exponentially small in $H$, this implies that
its complement, the set $(H_{2g})^c$, all of whose elements correspond
to $3$-manifolds which are hyperbolic and have Heegaard genus equal to
$g$, is exponentially large in $H$.

Now we are ready to put the parts together and obtain the main theorem.

\begin{proof}(of Theorem \ref{theorem:main}) %
Consider $\vK$, which is a complete subgroup of $\MCG(S_g)$. By
Theorem \ref{Hg}, there is a finitely generated subgroup $H_0 <
\MCG(S_g)$ such that for any finitely supported probability
distribution $\mu$ whose support generates a subgroup $H$ containing
$H_0$, the splitting distance $\dsp(M(w_n))$ grows linearly with
exponential decay.

We shall consider the subgroup $H$ generated by $\{ H_0 \cup \tau \}$,
where $\tau \in \vK$ is the element previously constructed in Lemma
\ref{lemma:tau}, for which $\lambda(\tau) = 1$. This ensures that the
homomorphism $\lambda \colon H \to \Z$ is surjective. The subgroup $H$
is finitely generated, as $H_0$ is finitely generated. We may now
choose a symmetric random walk, supported on a finite generating set
for $H$, and the image of this random walk under the homorphism
$\lambda$ is a symmetric finite range random walk on
$\Z$. Furthermore, we may assume that the image of the random walk on
$\Z$ is irreducible, for example by adding an element to the
generating set which maps to $0$.

We shall write $H_{2g}$ to denote the subset of $H$ consisting of
group elements which give rise to Heegaard splittings of distance less
than or equal to $2g$.  Suppose some $k \in \Z$ is not achieved as the
Casson invariant of a hyperbolic homology sphere with Heegaard genus
$g$, then in particular we have
\[ \lambda^{-1}(k) \subseteq H_{2g}. \]
The set on the right hand side is exponentially small by Theorem
\ref{Hg}.

Now since $\lambda$ is a homomorphism on $H$, it projects the random
walk on $H$ to an irreducible Markov process on $\Z$, and for a
symmetric finite range random walk, at step $n$,
$$ \P( \lambda(w_n) = k) \ge c / \sqrt{n} $$
with some small number $c$, see for example Lawler and Limic
\cite{LL}*{Section 2.1}. This contradicts the assumption of
$\lambda^{-1}(k)$ being exponentially small.
\end{proof}

Let us now sketch the proof of Theorems \ref{Hg} and \ref{theorem:2}.
We now give a brief overview of the argument showing that the Heegaard
splitting distance $\dsp(M(w_n))$ grows linearly with exponential
decay.

Let $x_0$ be a basepoint in the complex of curves $\mathcal{C}(S_g)$,
and consider the image of the random walk under the orbit map $w_n
\mapsto w_n x_0$.  A key intermediate result is to show that for a
random walk on a suitably chosen subgroup $H$ of $\vK$, the distance
of the sample path from the disc set, $d(\mathcal{D}, w_n x_0)$, grows
linearly in $n$ with exponential decay, and we now give an outline of
the argument for this result, omitting certain technical details.

We shall write $\N_0$ to denote the set of non-negative integers.
Consider the sequence of random variables $X_n = d(\mathcal{D}, w_n
x_0)$ with values in $\N_0$. This is not a Markov chain, but the
conditional probabilities
%
%
\[ \P( X_{n+1} = j \mid w_n = \phi ) \]
are well-defined. We shall show the following pair of ``local''
estimates for these conditional probabilities. We state approximate
versions of the properties here, and precise versions later on.

\begin{enumerate}

\item If the sample path location $w_n x_0$ is reasonably far from the
disc set, then the probability that after $m$ steps you have moved a
distance $r$ closer to the disc set decays exponentially in $r$,
i.e. there is some $q < 1$ such that
  \[ \P( X_{n+m} = t - r \mid w_n = \phi \text{ with } d(\mathcal{D},
  \phi x_0) = t) \le q^{r+1}. \]
\item If the sample path location $w_n x_0$ is close to the disc set,
then there is a definite chance $\e > 0$ that after $m$ steps you have
moved a reasonably distance $r$ away from the disc set, i.e.
  \[ \P( X_{n+m} \ge r \mid w_n = \phi \text{ with } d(\mathcal{D}, \phi
  x_0) = 0 ) \ge \e.  \]

\end{enumerate}

We briefly indicate some of the details that need to be addressed. We
need to obtain the estimates above for all $n$ and some fixed but
suitably large $m$. It is also more convenient to work with a coarse
version of distance, in which we choose $X_n$ to be the integer part
of $d(\mathcal{D}, w_n x_0)/R$, for some reasonably large $R$, rather
than $d(\mathcal{D}, w_n x_0)$ itself.

In Section \ref{section:coarse} we review some basic results in coarse
geometry, before using these to show the first property in Section
\ref{section:local far}.  Then in Section \ref{section:train track} we
review some basic results about train tracks and shadow sets, and then
use these to show the second property in Section \ref{section:local
  near}. 

In Section \ref{section:distance from D} we show how to use the two
properties above to show that $d(\mathcal{D}, w_n x_0)$ grows
linearly, with exponential decay, and we now describe our approach.

The sequence of random variables $\{ X_n \}$ does not arise from a
Markov chain, but we can compare the distributions of the $\{ X_n \}$
with the distributions $\{ Y_n \}$ arising from a Markov chain on
$\N_0$, which never increases by more than one unit per step, and has
transition probabilities given by $p(0, 0) = 1 - \e$, and $p(t, t - r)
= q^{r+1}$, for $r \ge 0$. The first few vertices of this Markov chain
are illustrated in Figure \ref{pic:markov-chain}.

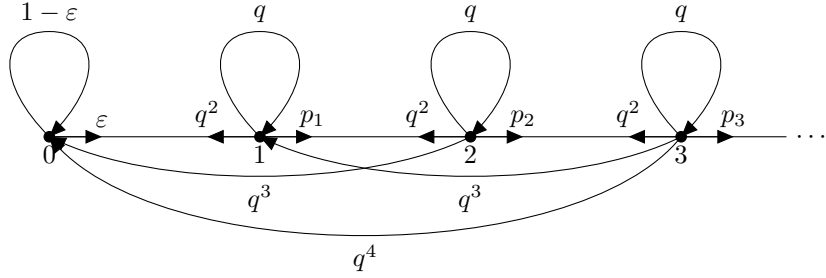
\begin{figure}[H] \begin{center}
\begin{tikzpicture}[scale=1.4]

\draw (0, 0) -- (7,0) node [right] {$\cdots$};

\filldraw[black] (0,0) circle (0.05cm) node [below] {$0$};
\filldraw[black] (2,0) circle (0.05cm) node [below] {$1$};
\filldraw[black] (4,0) circle (0.05cm) node [below] {$2$};
\filldraw[black] (6,0) circle (0.05cm) node [below] {$3$};

\draw[arrows=-triangle 45] (0,0) -- (0.5,0) node [above] {$\e$};
\draw[arrows=-triangle 45] (2,0) -- (1.5,0) node [above] {$q^2$};

\draw[arrows=-triangle 45] (0,0) .. controls (-0.5, 0.5) and (-0.5, 1) .. (0,1) node [above] {$1 - \e$}
                                 .. controls (0.5, 1) and (0.5, 0.5) .. (0, 0);

\begin{scope}[xshift=2cm]
\draw[arrows=-triangle 45] (0,0) -- (0.5,0) node [above] {$p_1$};
\draw[arrows=-triangle 45] (2,0) -- (1.5,0) node [above] {$q^2$};

\draw[arrows=-triangle 45] (2,0) .. controls (1.5, 0.5) and (1.5, 1) .. (2,1) node [above] {$q$}
                                 .. controls (2.5, 1) and (2.5, 0.5) .. (2, 0);

\draw[arrows=-triangle 45] (4,0) .. controls (3, -0.5) and (1, -0.5) .. (0, 0) node [midway, below] {$q^3$};

\end{scope}

\draw[arrows=-triangle 45] (4,0) .. controls (3, -0.5) and (1, -0.5) .. (0, 0) node [midway, below] {$q^3$};

\begin{scope}[xshift=4cm]
\draw[arrows=-triangle 45] (0,0) -- (0.5,0) node [above] {$p_2$};
\draw[arrows=-triangle 45] (2,0) -- (1.5,0) node [above] {$q^2$};

\draw[arrows=-triangle 45] (2,0) .. controls (1.5, 0.5) and (1.5, 1) .. (2,1) node [above] {$q$}
                                 .. controls (2.5, 1) and (2.5, 0.5) .. (2, 0);

\end{scope}

\draw[arrows=-triangle 45] (6,0) -- (6.5,0) node [above] {$p_3$};

\draw[arrows=-triangle 45] (2,0) .. controls (1.5, 0.5) and (1.5, 1) .. (2,1) node [above] {$q$}
                                 .. controls (2.5, 1) and (2.5, 0.5) .. (2, 0);

\draw[arrows=-triangle 45] (6,0) .. controls (5, -1.25) and (1, -1.25) .. (0, 0) node [midway, below] {$q^4$};

\end{tikzpicture}
\end{center} 
\caption{The Markov chain $(\N_0, P)$, where $p_i = 1  - q - q^2 - \cdots - q^{i+1}$.} 
\label{pic:markov-chain}
\end{figure}

Intuitively, it is always more likely that the random variables $\{
X_n \}$ move further to the right than the random variables $\{ Y_n
\}$ arising from the Markov chain. This means that the distribution of
each $Y_n$ will have greater weight on small values than the
distribution for the corresponding $X_n$, or more precisely
$F_{X_n}(t) \le F_{Y_n}(t)$ for all $t$, where $F_{X_n}$ and $F_{Y_n}$
are the cumulative probability functions for $X_n$ and $Y_n$
respectively.  This property is usually described by saying that the
random variables $X_n$ \emph{stochastically dominate} the random
variables $Y_n$, written as $Y_n \prec X_n$. A standard argument from
the theory of Markov chains shows that the random variables $Y_n$
arising from the Markov chain satisfies the linear progress with
exponential decay property that we require, and so this implies the
linear progress with exponential decay property for the $X_n$.

Finally, in Section \ref{section:splitting distance} we use some more
coarse geometry to extend this to show that $d(\mathcal{D}, w_n
h_{S^3} \mathcal{D})$ grows linearly with exponential decay.

\section{Coarse geometry} \label{section:coarse}

Let $(X, d)$ be a $\delta$-hyperbolic space, which need not be locally
compact.  Recall that the Gromov product of two points $y$ and $z$ in
$X$, with respect to a basepoint $x \in X$, is defined to be
\[ \gp{x}{y}{z} = \tfrac{1}{2}( d(x, y) + d(x, z) - d(y, z)
).  \]
This is equal to the distance from $x$ to a geodesic from $y$ to $z$
up to bounded additive error.  Given a basepoint $x \in X$, and a
number $R \ge 0$, we define the shadow set of a point $y \in X$, with
parameter $R$, to be
\begin{equation} \label{eq:shadow def} %
S_{x}(y, R) = \{ z \in X \mid \gp{x}{y}{z} \ge d(x, y) - R \}.  
\end{equation}

We start with an elementary observation concerning distances in
$\delta$-hyperbolic spaces. Suppose you travel from a point $x$ to a
point $y$ along a geodesic $\gamma$, and then from the point $y$ to
$z$ along a geodesic $\gamma'$. By thin triangles, the two geodesics
$\gamma$ and $\gamma'$ fellow travel for some distance near $y$,
before moving apart, and we can think of the length of the
fellow-travelling segments as measuring the overlap of the two
geodesics.

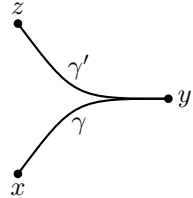
\begin{figure}[H] \begin{center}
\begin{tikzpicture}[scale=0.5]

\draw[thick] (0, 0) .. controls (1.5, 2) .. node [below] {$\gamma$} (4, 2);
\filldraw (0, 0) circle (0.1) node [below] {$x$};
\filldraw (4, 2) circle (0.1) node [right] {$y$};

\draw[thick] (4, 2) .. controls (1.5, 2) .. node [above] {$\gamma'$} (0, 4);
\filldraw (0, 4) circle (0.1) node [above] {$z$};

\end{tikzpicture}
\end{center} 
\caption{Overlapping geodesics} 
\label{pic:overlap}
\end{figure}

The total distance from $d(x, z)$ is equal to $d(x, y) + d(y, z)$
minus approximately twice this overlap. It is convenient to use the
Gromov product $\gp{y}{x}{z}$ as a measure of this overlap, and then
it follows immediately from the definition of the Gromov product that
\[  d(x, z) = d(x, y) + d(y, z) - 2 \gp{y}{x}{z},   \]
i.e. we can estimate the distance from $x$ to $z$ as the sum of the
distances from $x$ to $y$ and $y$ to $z$, together with a term
involving the Gromov product.

The aim of this section is to show similar estimates for the case in
which the point $x$ is replaced by a quasiconvex set, and also for the
case in which both the points $x$ and $z$ are replaced with
quasiconvex sets, as illustrated in Figure \ref{pic:quasiconvex2},
where $x$ is the closest point on $D$ to $y$ and $z$ is the closest
point on $E$ to $y$.

\begin{figure}[H] \begin{center}
\begin{tikzpicture}[scale=0.5]

\def\partt{ 
\draw[thick] (0, 0) .. controls (1.5, 2) .. node [below] {$\gamma$} (4, 2);
\filldraw (0, 0) circle (0.1) node [below] {$x$};
\filldraw (4, 2) circle (0.1) node [right] {$y$};

\draw[thick] (4, 2) .. controls (1.5, 2) .. node [above] {$\gamma'$} (0, 4);
\filldraw (0, 4) circle (0.1) node [above] {$z$};

\begin{scope}[xshift=-0.707cm, yshift=-0.707cm]
        \draw[rotate=45] (0, 0) ellipse (1 and 2) node {$D$};
\end{scope}

\filldraw (-1.2, 0.8) circle (0.1) node [below] {$x'$};
}

\partt

\begin{scope}[xshift=10cm]

        \partt

        \begin{scope}[xshift=-0.707cm, yshift=4.707cm]
                \draw[rotate=-45] (0, 0) ellipse (1 and 2) node {$E$};
        \end{scope}

        \filldraw (0.85, 5.5) circle (0.1) node [right] {$z'$};

\end{scope}

\end{tikzpicture}
\end{center} 
\caption{Estimating distances between quasiconvex sets and points.} 
\label{pic:quasiconvex2}
\end{figure}
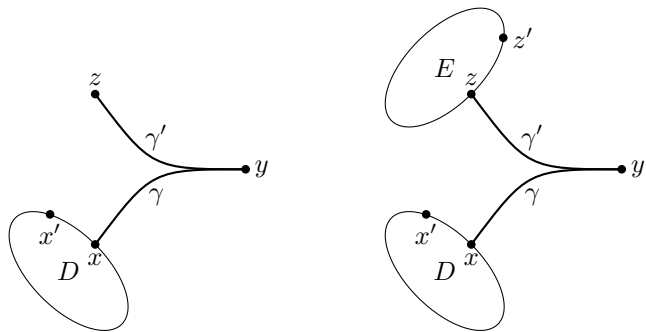

Recall that a set $D \subset X$ is \emph{$Q$-quasiconvex} if any
geodesic connecting two points of $D$ is contained in a
$Q$-neighbourhood of $D$. We now state the two results we will use in
the later sections. In the first result, we replace $x$ with a
quasiconvex set $D$, and find a lower bound on $d(D, z)$, using an
additional hypothesis on $\gp{y}{x}{z}$. Furthermore, we show that the
Gromov product based at $y$ of $z$ with any point in $D$ is equal to
$\gp{y}{x}{z}$ up to bounded additive error depending only on $\delta$
and $Q$.

\begin{prop} \label{prop:one quasiconvex} %
Let $(X, d)$ be a $\delta$-hyperbolic space, which need not be locally
compact. Given a number $Q$, there are numbers $A, B$ and $C$, which
only depend on $\delta$ and $Q$, such that if $D$ is a $Q$-quasiconvex
set, and $x$ is a closest point in $D$ to $y$, then for any point $z$
with
\begin{equation} \label{eq:prop:one quasiconvex1}   
\gp{y}{x}{z} \le d(x, y) - A,  
\end{equation}
then
\begin{equation} \label{eq:prop:one quasiconvex2}
 d(D, z) \ge d(x, y) + d(y, z) - 2 \gp{y}{x}{z} - B.  
\end{equation}
Furthermore, for any point $x' \in D$,
\begin{equation} \label{eq:prop:one quasiconvex3}
\norm{ \gp{y}{x}{z} - \gp{y}{x'}{z} } \le C.    
\end{equation}
\end{prop}

For the second result, we produce a similar estimate with both $x$ and
$z$ replaced with quasiconvex sets $D$ and $E$, with additional
hypotheses on $\gp{y}{x}{z}$.  Furthermore, we show that the Gromov
product based at $y$ of any point in $D$ with any point in $E$, is
equal to $\gp{y}{x}{z}$, again up to bounded additive error depending
only on $\delta$ and $Q$.

\begin{prop} \label{prop:two quasiconvex} %
Let $(X, d)$ be a $\delta$-hyperbolic space, which need not be locally
compact. Given a number $Q$, there are numbers $A, B$ and $C$, which
only depend on $\delta$ and $Q$, such that if $D$ and $E$ are
$Q$-quasiconvex sets, and $x$ is a closest point in $D$ to $y$, and
$z$ is a closest point in $E$ to $y$, then if
\begin{equation} \label{eq:prop:two quasiconvex1}   
\gp{y}{x}{z} \le \min \{ d(x, y),  d(y, z) \} - A,  
\end{equation}
then
\begin{equation} \label{eq:prop:two quasiconvex2}   
 d(D, E) \ge d(x, y) + d(y, z) - 2 \gp{y}{x}{z} - B.  
\end{equation}
Furthermore, for any points $x' \in D$ and $z' \in E$,
\begin{equation} \label{eq:prop:two quasiconvex3}   
\norm{ \gp{y}{x}{z} - \gp{y}{x'}{z'} } \le C.    
\end{equation}
\end{prop}

These results follow from standard arguments in coarse geometry, but
we give full details for the convenience of the reader. We recall the
following result regarding approximate trees in $\delta$-hyperbolic
spaces, see for example Ghys and de la Harpe \cite{gh}*{Section 2.2}.

\def\KT{K_T}

\begin{lemma}[Approximate tree]\label{approximate_tree_lemma} %
Let $(X, d)$ be a $\delta$-hyperbolic space, which need not be locally
compact.  Then there is a number $\KT$, which depends only on
$\delta$, such that for any finite collection of points $x_1, \ldots,
x_n$, there is a geodesic tree $T$ in $X$ containing the $x_i$, such
that
\begin{equation} \label{eq:atree}
d_T(x, y) - \KT n  \le d(x, y) \le d_T(x, y)
\end{equation}
for all points $x$ and $y$ in $T$.
We call $T$ the {\em approximate tree} determined by the $x_i$.
\end{lemma}

As the Gromov product is defined in terms of distances, if we write
$\gp{x}{y}{z}^T$ for the Gromov product in the tree $T$ with vertices
$\{ x_1, \ldots x_n \}$, then $\norm{ \gp{x}{y}{z}^T - \gp{x}{y}{z} }$
is bounded by a number, in fact $3 \KT n /2$, which only depends on
$\delta$ and $n$, where $n$ is the number of vertices in the tree.

Recall that any three distinct points in a tree determine a unique
\emph{center}, namely the unique point that lies in the intersection
of the three geodesics connecting the three possible pairs of
points. It follows from the definition of the Gromov product that for
any points $x, y$ and $z$ in a tree $T$, if $v$ is the center for $x,
y$ and $z$, then $\gp{z}{x}{y}^T = d_T(z, v)$.

Let $I$ be a connected subinterval of $\R$.  We say a path $\gamma
\colon I \to X$ is a \emph{$(K, c)$-quasigeodesic} if
\[ \tfrac{1}{K} \norm{x - y} - c \le d( \gamma(x), \gamma(y) ) \le K
\norm{x - y} + c, \]
for all $x$ and $y$ in $\R$. In a $\delta$-hyperbolic space,
quasigeodesics are contained in bounded neighbourhoods of
geodesics. This is often known as the Morse Lemma, see for example
Bridson and Haefligger \cite{bh}*{Theorem III.H.1.7}. 

\begin{lemma} \label{lemma:morse} %
Let $\gamma$ be a $(K, c)$-quasigeodesic in a $\delta$-hyperbolic
space. Then there is a number $L$, depending only on $K, c$ and
$\delta$, such that $\gamma$ is contained in an $L$-neighbourhood of
the geodesic connecting its endpoints.
\end{lemma}

As geodesics in the approximate tree $T$ are $(1, \KT
n)$-quasigeodesics, this implies that any geodesic in $T$ is contained
in an $L_n$-neighbourhood of the geodesic in $X$ connecting its
endpoints, for some number $L_n$, depending only on $\delta$ and
$n$.

\begin{prop} \label{prop:atree3} %
Let $(X, d)$ be a $\delta$-hyperbolic space, which need not be locally
compact.  Let $D$ be a $Q$-quasiconvex set, and let $y$ be a point in
$X$. Let $x$ be the closest point in $D$ to $y$, and let $x'$ be an
arbitrary point in $D$. Let $T$ be an approximate tree determined by a
set of $n$ points, which include $x, x'$ and $y$, and let $v$ be the
center of $x, x'$ and $y$ in $T$. Then there is a number $A$, which
only depends on $\delta, Q$ and $n$, such that 
\begin{equation} \label{eq:prop:atree}
d(x, v) \le A.
\end{equation}
\end{prop}

Although the constant $A$ depends on the number of points determining
the approximate tree $T$, we will only consider approximate trees
determined by at most $5$ points.

\begin{proof}(of Proposition \ref{prop:atree3}.) %
Let $D$ be a $Q$-quasiconvex set, let $y$ be a point in $X$, let $x$
be the closest point in $D$ to $y$, and let $x'$ be any other point in
$D$.  Let $T$ be an approximate tree containing $n$ points, which
include $x, x'$ and $y$, and let $v$ be the center of $x, x'$ and $y$
in $T$. In particular, the minimal subtree in $T$ containing $x, x'$
and $y$ also contains $v$, and has the configuration illustrated in
Figure \ref{pic:atree3}, possibly with edges of zero length.

\begin{figure}[H] \begin{center}
\begin{tikzpicture}[scale=0.5]

\draw[thick] (0, 0) -- (2, 2);
\filldraw (0, 0) circle (0.1) node [left] {$x$};
\filldraw (2, 2) circle (0.1) node [left] {$v$};

\draw[thick] (2, 2) -- (4, 2);
\filldraw (4, 2) circle (0.1) node [right] {$y$};

\draw[thick] (2, 2) -- (0, 6);
\filldraw (0, 6) circle (0.1) node [left] {$x'$};

\end{tikzpicture}
\end{center} 
\caption{A minimal subtree containing $x$, $x'$ and $y$,
  with center $v$.}
\label{pic:atree3}
\end{figure}
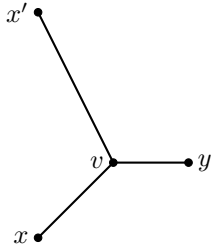

By geodesic stability, Lemma \ref{lemma:morse}, as $v$ lies on the
geodesic in $T$ connecting $x$ and $x'$, the point $v$ lies within a
bounded distance $L_n$ of the geodesic in $X$ connecting $x$ and
$x'$. As $D$ is $Q$-quasiconvex, the geodesic in $X$ from $x$ to $x'$
is contained in a $Q$-neighbourhood of $D$, and so
\begin{equation} \label{eq:containment1} 
d(D, v) \le Q + L_n. 
\end{equation}  
As $x$ is a closest point in $D$ to $y$, $d(x, y) = d(D, y)$, and so
the triangle inequality implies
\[  d(x, y) \le d(D, v) + d(v, y). \]
Using \eqref{eq:containment1} this implies
\[  d(x, y) \le Q + L_n + d(v, y). \]
As the metric in $X$ is coarsely equivalent to the metric in $T$,
using \eqref{eq:atree} implies
\[  d_T(x, y) - n \KT \le Q + L_n + d_T(v, y). \]
As $T$ is a tree, $d_T(x, y) = d_T(x, v) + d_T(v, y)$, which gives
\begin{equation} \label{eq:containment2} %
d(x, v) \le d_T(x, v) \le A,
\end{equation}
where $A = Q + L_n + n \KT$, and so the distance from $x$ to $v$ in
$X$, is bounded by a number which only depends on $\delta, Q$ and $n$,
the number of points in the approximate tree, as required.
\end{proof}

We now prove Proposition \ref{prop:one quasiconvex}.

\begin{proof}
Let $T$ be an approximate tree containing the points $x, x', y$ and
$z$, and let $v$ be the center in $T$ for $x, x'$ and $y$, and let $w$
be the center in $T$ for $x, y$ and $z$.  The set $D$ and the points
$x, x'$ and $y$ satisfy the conditions of Proposition
\ref{prop:atree3}, and so $d(x, v) \le A_1$, where $A_1$ depends only on
$\delta, Q$ and $n$, and in this case $n$ is equal to $4$. 

We now show that condition \eqref{eq:prop:one quasiconvex1}, where $A$
is chosen to be $6 \KT + A_1$, implies that the centers $v$ and
$w$ occur in that order along the geodesic from $x$ to $y$, so the
approximate tree $T$ has the configuration illustrated in Figure
\ref{pic:atree4}.

\begin{figure}[H] \begin{center}
\begin{tikzpicture}[scale=0.5]

\draw[thick] (0, 0) node [left] {$x$} -- (2, 2) node [left] {$v$};

\draw[thick] (2, 2) -- (4, 2) node [right] {$w$};

\draw[thick] (2, 2) -- (0, 5) node [left] {$x'$};

\draw[thick] (4, 2) -- (6, 0) node [right] {$y$};

\draw[thick] (4, 2) -- (6, 7) node [right] {$z$};

\filldraw (0, 0) circle (0.1);
\filldraw (2, 2) circle (0.1);
\filldraw (4, 2) circle (0.1);
\filldraw (0, 5) circle (0.1);
\filldraw (6, 0) circle (0.1);
\filldraw (6, 7) circle (0.1);

\end{tikzpicture}
\end{center} 
\caption{An approximate tree containing $x, x', y$ and $z$.}
\label{pic:atree4}
\end{figure}
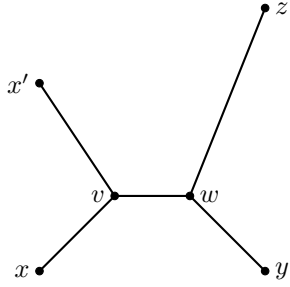

As the metric in $T$ is equal to the metric in $X$ up to bounded
additive error, condition \eqref{eq:prop:one quasiconvex1} is
equivalent to 
\[ \gp{y}{x}{z}^T \le d_T(x, y) - A + 6 \KT. \]
In $T$, the Gromov product $\gp{y}{x}{z}^T$ is equal to $d_T(y, w)$,
and the center $v$ lies on the geodesic connecting $x$ to $y$, so
\[ d_T(w, y) \le d_T(x, v) + d_T(v, y) - A + 6 \KT . \]
As $d(x, v) \le A_1$, this implies
\[ d_T(w, y) \le d_T(v, y) - A + 6 \KT + A_1. \]
As we have chosen $A = 6 \KT  + A_1$, this implies that $d_T(w, y)
\le d_T(v, y)$, and so the centers $v$ and $w$ must lie in that order
along the geodesic from $x$ to $y$. The constant $A$ only depends on
$\delta, Q$ and $n$, which in this case is $4$, and so $A$ only
depends on $\delta$ and $Q$, as required.

In order to show \eqref{eq:prop:one quasiconvex2}, let $x'$ be the
closest point on $D$ to $z$. Then, given the configuration of the
approximate tree shown in Figure \ref{pic:atree4},
\[ d_T(x', z) \ge d_T(v, z).  \]
As $d_T(x, v) \le A_1$,
\[ d_T(x', z) \ge d_T(x, z) - A_1,  \]
and now using the definition of the Gromov product,
\[ d_T(x', z) \ge d_T(x, y) + d_T(y, z) - 2 \gp{y}{x}{z}^T - A_1.  \]
As the metrics in $T$ and $X$ are equal up to additive constants, this
implies that
\[ d(D, z) \ge d(x, y) + d(y, z) - 2 \gp{y}{x}{z} - B,  \]
where $B = A_1 + 14 \KT$, which only depends on $\delta$ and $Q$, as
required.

Finally, the configuration of the approximate tree shown in Figure
\ref{pic:atree4}, shows that the two Gromov products $\gp{y}{x}{z}^T$
and $\gp{y}{x'}{z}^T$ are both equal to $d_T(y, w)$, and so as the
metrics in $T$ and $X$ are equal up to additive error, this shows
\eqref{eq:prop:one quasiconvex3}, for some constant $C$, which only
depends on $\delta$ and $Q$, as required.
\end{proof}

Finally, we prove Proposition \ref{prop:two quasiconvex}.

\begin{proof}
Consider an approximate tree $T$ containing the points $x, x', y, y'$
and $z$.  Let $v$ be the center in $T$ for $x, x'$ and $y$, let
$w$ be the center in $T$ for $z, z'$ and $y$, and let $u$ be the
center for $x, y$ and $z$.

The set $D$ and the points $x, x'$ and $y$ satisfy the conditions of
Proposition \ref{prop:atree3}, and so $d(x, v) \le A_1$, where $A_1$
depends only on $\delta, Q$ and $n$, and in this case $n$ is equal to
$5$. Similarly, the set $E$ and the points $z, z'$ and $y'$ also
satisfy the conditions of Proposition \ref{prop:atree3}, and so $d(z,
w) \le A_1$, where $A_1$ depends only on $\delta, Q$ and $n$, and
again in this case $n$ is equal to $5$.

We now show that condition \eqref{eq:prop:two quasiconvex1}, where $A$
is chosen to be $15 \KT / 2 + A_1$, implies that the centers $v, u$
and $w$ occur in that order along the geodesic from $x$ to $z$, so the
approximate tree $T$ has the configuration illustrated in Figure
\ref{pic:approximate-tree2}, possibly with some zero length edges.

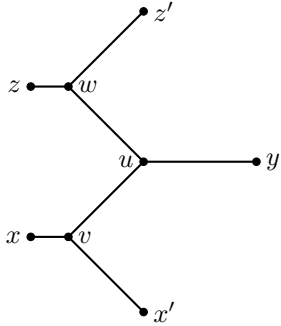
\begin{figure}[H] \begin{center}
\begin{tikzpicture}[scale=0.5]

\draw[thick] (0, 0) -- (-1, 0) node [left] {$x$};

\draw[thick] (0, 0) node [right] {$v$} -- (2, 2) node [left] {$u$};

\draw[thick] (0, 0) -- (2, -2) node[right] {$x'$};

\draw[thick] (2, 2) -- (5, 2) node [right] {$y$};

\draw[thick] (2, 2) -- (0, 4) node [right] {$w$};

\draw[thick] (0, 4) -- (-1, 4) node [left] {$z$};

\draw[thick] (0, 4) -- (2, 6) node[right] {$z'$};

\filldraw (0, 0) circle (0.1);
\filldraw (-1, 0) circle (0.1);
\filldraw (2, -2) circle (0.1);
\filldraw (5, 2) circle (0.1);
\filldraw (0, 4) circle (0.1);
\filldraw (-1, 4) circle (0.1);
\filldraw (2, 6) circle (0.1);
\filldraw (2, 2) circle (0.1);

\end{tikzpicture}
\end{center} 
\caption{An approximate tree for $x, x', y, z$ and $z'$.}
\label{pic:approximate-tree2}
\end{figure}

As the metric in $T$ is equal to the metric in $X$ up to bounded
additive error, condition \eqref{eq:prop:two quasiconvex1} is
equivalent to 
\[ \gp{y}{x}{z}^T \le \min \{ d_T(x, y), d_T(y, z) \} - A + 15
\KT/2. \]
In $T$, the Gromov product $\gp{y}{x}{z}^T$ is equal to $d_T(y, u)$,
and the center $v$ lies on the geodesic in $T$ connecting $x$ to $y$,
and the center $w$ lies on the geodesic connecting $z$ to $y$, so
\[ d_T(y, u) \le \min \{ d_T(x, v) + d_T(v, y), d_T(z, w) + d_T(w, y)
\} - A + 15 \KT / 2. \]
As $d(x, v) \le A_1$, and $d(z, w) \le A_1$, this implies
\[ d_T(y, u) \le \min \{ d_T(v, y), d_T(w, y) \} - A + 15 \KT / 2 + A_1. \]
As we have chosen $A = 15 \KT / 2 + A_1$, this implies that $d_T(y, u)
\le d_T(v, y)$, and also $d_T(y, u) \le d_T(w, y)$, and so the centers
$v$ and $w$ must lie further from $y$ than $u$, and the only way in
which this can happen is if the center $u$ lies between $v$ and $w$,
as illustrated in Figure \ref{pic:approximate-tree2}.  The constant
$A$ only depends on $\delta, Q$ and $n$, which in this case is $5$,
and so $A$ only depends on $\delta$ and $Q$, as required.

For estimate \eqref{eq:prop:two quasiconvex2}, the distance $d(D, E)$
between the two quasiconvex sets, suppose that $x'$ and $z'$ are
closest points in $D$ and $E$ respectively, i.e. $d(x', z') = d(D,
E)$. In the approximate tree, $d_T(x', z') \ge d_T(x, z)$. Using the
definition of the Gromov product in the approximate tree gives
\[  d_T(x', y') \ge d_T(x, y) + d_T(y,z) - 2 \gp{y}{x}{z}^T - 2A_1.    \]
Now using Lemma \ref{approximate_tree_lemma} to estimate distances in
$X$, gives
\[ d(D, E) \ge d(x, y) + d(y,z) - 2 \gp{y}{x}{z} - 2A_1 -5 \KT.  \]
If we choose $B = 2A_1 + 5 \KT$, which only depends on $\delta$ and
$Q$, then this shows \eqref{eq:prop:two quasiconvex2}.

For the final estimate for the Gromov products, observe that in the
approximate tree $T$, the Gromov product $\gp{y}{x}{z}^T$ is equal to
$ \gp{y}{x'}{z'}^T$, and so again by Lemma
\ref{approximate_tree_lemma},
\[ \norm{ \gp{y}{x}{z} - \gp{y}{x'}{z'} } \le 15 \KT /2, \]
and so we may choose $C = 15 \KT /2$, which only depends on $\delta$
and $Q$, and this gives the final estimate \eqref{eq:prop:two quasiconvex3}.
\end{proof}

\section{Local estimates far from the disc set} \label{section:local far}

We are interested in estimating distances from points in the curve
complex to a particular disc set $\mathcal{D}$. Given a positive
number $R$, it will be convenient to consider the following function
$\phi_R \colon \mathcal{C}(S_g) \to \N_0$ defined by
\[ \phi_R(x) = \lfloor d(\mathcal{D}, x) / R \rfloor, \]
where for any real number $r$, the function $\lfloor r \rfloor$ is the
largest integer less than or equal to $r$.  In particular, a random
walk on the mapping class group gives rise to a sequence of random
variables $X_n = \phi_R(w_n x_0)$, with values in $\N_0$.

We now show that if $X_m$ is sufficiently large, then the probability
that $X_{m+n}$ is less than $X_m - r$ decays exponentially in $r$, for
all $n$ sufficiently large. As the random walk is a Markov chain on
$G$, it suffices to consider the case of a random walk of length $n$
starting at some point $g x_0$ in $\mathcal{C}(S_g)$.

\begin{prop} \label{prop:backtrack} %
Let $\mu$ be a finitely supported probability distribution on the
mapping class group, whose semi-group support $\langle \supp(\mu)
\rangle_+$ is a non-elementary subgroup.  For any number $q > 0$,
there are numbers $R$ and $N$, which depend on $q$ and $\mu$, such
that
\[ \P( \phi_R( g w_n x_0 ) \le t + 1 - r \mid \phi_R( g x_0 ) = t \ge 1
) \le q^{r}, \]
for all $n \ge N$ and $r \ge 0$.
\end{prop}

Recall that the probability distribution $\mu$ depends on the group
$G$, and so any number which depends on $\mu$ implicitly depends on
the group $G$, and hence on the coarse geometry constants determined
by $G$.

In order to show this, we make use of Proposition \ref{prop:one
  quasiconvex}, which enables us to estimate the distance of $g w_n
x_0$ from a quasiconvex set in terms of the distance travelled by the
sample path, and an estimate on the size of a particular Gromov
product. We will also use the fact that the distance travelled by the
sample path in the curve complex $\mathcal{C}(S_g)$ grows linearly in
$n$ with exponential decay.

\def\Llin{L}
\def\Klin{K_\ell}
\def\clin{c_\ell}

\begin{theorem} \cite{M2} \label{theorem:linear progress} %
Let $\mu$ be a probability distribution on the mapping class group
with finite support whose semi-group support $\langle \supp(\mu)
\rangle_+$ is a non-elementary subgroup. Then there are numbers
$\Klin, \Llin > 0$ and $\clin < 1$, which depend on $\mu$, such that
\[ \P( d( w_n x_0, x_0) \le \Llin n ) \le \Klin \clin^n \]
for all $n$.
\end{theorem}

In particular, this result holds when $n=0$, and so the number $\Klin$
must be at least $1$. In order to estimate probabilities involving
Gromov products, we will also use the following estimate for the
probability a sample path lies in a shadow set, as defined in
\eqref{eq:shadow def}, which we shall refer to as exponential decay
for shadows.

\def\Ksh{K_S}
\def\csh{c_S}

\begin{lemma} \cite{M2} \label{lemma:shadow decay} %
Let $\mu$ be a probability distribution on the mapping class group
with finite support whose semi-group support $\langle \supp(\mu)
\rangle_+$ is a non-elementary subgroup of the mapping class
group. Then there are numbers $\Ksh > 0$ and $\csh < 1$, which depend
on $\mu$, such that for any point $y$,
\[ \P ( w_n x_0 \in S_{x_0}(y, R) ) \le \Ksh \csh^{d(x_0, y) - R}, \]
for all $n$.
\end{lemma}

The following estimate for the Gromov product then follows immediately
from the definition of shadow sets, \eqref{eq:shadow def}.
\[ \P \left( \ \gp{x_0}{y}{w_n x_0} \ge d(x_0, y) - R \ \right) \le K
c^{d(x_0, y) - R}. \]

In order to simplify notation, we shall unify the exponential decay
constants in the two results above, by setting $K = \max \{ \Klin,
\Ksh \}$ and $c = \max \{ \clin, \csh \}$. 

We now give a brief overview of the proof. Consider a random walk of
length $n$ starting at a point $g x_0$, reasonably far from the disc
set $\mathcal{D}$, and so the endpoint of the random walk is $g w_n
x_0$. By linear progress, it is very likely that the random walk has
gone a reasonable distance, and by exponential decay for shadows, it
is very unlikely that the Gromov product $\gp{g x_0}{x_0}{g w_n x_0}$
is large, and this Gromov product measures how much the geodesic from
$g x_0$ to $g w_n x_0$ fellow-travels, or backtracks, along the path
from $g x_0$ to $x_0$. In this case, we may then apply Proposition
\ref{prop:one quasiconvex} to estimate the distance from $\mathcal{D}$
to $g w_n x_0$. This is the main case we consider in the proof below,
but we also need to estimate separately the less likely cases in which
the random walk does not go very far, or the Gromov product is so
large that Proposition \ref{prop:one quasiconvex} does not apply. The
final bounds arise from adding the bounds we get in each case.

\begin{proof} (of Proposition \ref{prop:backtrack}.) %
We need to choose appropriate values for $N$ and $R$. In order to make
clear that there is no circularity in our choice of constants, we now
state how we will choose $N$ and $R$. We shall choose
\begin{equation} \label{eq:backtrackR} 
R \ge \max \{ \log(q^2/3 K) / \log(c) + A + C, 2 \log(q^2 / 3 K)
/ \log(c) , B + 2C  \},  
\end{equation}
where $A, B$ and $C$ are the constants from Proposition \ref{prop:one
  quasiconvex}, which only depend on the coarse geometry constants
$\delta$ and $Q$, and $K$ and $c$ are the exponential decay constants,
which depend on $\mu$. We shall then choose
\begin{equation} \label{eq:backtrackN}
N \ge \max \{ 2R / L , \log (q^3 / 3 K) / \log(c) \},  
\end{equation}
with the same notation for constants as above, and where $L$ is the
linear progress constant, which only depends on $\mu$.  We emphasize
that our choice of $R$ only depends on $q$ and $\mu$, and our choice
of $N$ depends on $R$, $q$ and $\mu$.

We wish to estimate the probability that $\phi_R(g w_n x_0)$ takes
certain values, and so we need to estimate the distance from the disc
set $\mathcal{D}$ to $g w_n x_0$, and in the main case we consider we
shall do this by using Proposition \ref{prop:one quasiconvex}. We
shall apply Proposition \ref{prop:one quasiconvex} with $y = g x_0$,
$z = g w_n w_0$, and $x$ the nearest point in the disc set to $g
x_0$. Proposition \ref{prop:one quasiconvex} then says that the
following condition on the Gromov product
\begin{equation} \label{eq:gp condition} 
\gp{g x_0}{x}{g w_n x_0} \le d(x, g x_0) - A,   
\end{equation}
implies the following bound on the distance from $g w_n x_0$ to the
disc set $\mathcal{D}$,
\[  d(\mathcal{D}, g w_n x_0) \ge d(x, g x_0) + d(g x_0, g w_n x_0) - 2 \gp{g
  x_0}{x}{g w_n x_0} - B.    \]
Recall that as $x_0 \in \mathcal{D}$, the final part of Proposition
\ref{prop:one quasiconvex} shows that we may replace $\gp{g x_0}{x}{g
  w_n x_0}$ with $\gp{g x_0}{x_0}{g w_n x_0}$ up to bounded error, so
\eqref{eq:gp condition} is implied by the following condition
\[  \gp{g x_0}{x_0}{g w_n x_0} \le d(x, g x_0) - A - C.     \]
As $d(x, g x_0) = d(\mathcal{D}, g x_0)$, and $d(\mathcal{D}, g x_0)
\le R \phi_R(g x_0)$, and we have assumed $\phi_R(g x_0) = t$, we can
rewrite the condition above as
\begin{equation} \label{eq:condition1}  
\gp{g x_0}{x_0}{g w_n x_0} \le Rt - A - C.
\end{equation}

We shall consider various cases, depending on whether some combination
of condition \eqref{eq:condition1}, and the following condition
\eqref{eq:condition2}, hold. The second condition is that the sample
path has travelled a distance at least $2R$ from $g x_0$ to $g w_n
x_0$, i.e.
\begin{equation} \label{eq:condition2}
d(g x_0, g w_n x_0) \ge 2R.
\end{equation}

We shall consider the following three cases, defined in terms of the
conditions above, which cover all possibilities. The table below
summarizes the three cases, the various possibilities for $\phi_R( g
w_n x_0)$ which may occur given the conditions, and the bounds on the
probabilities that these values of $\phi_R(g w_n x_0)$ occur.

\bigskip

\begin{center}
\begin{tabular}{clll}
Case & Conditions & Value of $\phi_R(g w_n x_0)$ & Probability \\

1 & \eqref{eq:condition1} and \eqref{eq:condition2} hold & $\phi_R( g w_n x_0 ) \le t - r $ & $\le q^{r+1}/3, r \ge 1$ \\

  & & & $\le 1, r = 0 $ \\

2 & \eqref{eq:condition1} fails & $\phi_R( g w_n x_0 ) \ge 0$ & $\le q^{t+1}/3$ \\

3 & \eqref{eq:condition2} fails & $\phi_R( g w_n x_0 ) \ge t - 2$ & $\le q^3 / 3$  \\

\end{tabular}
\end{center}

\bigskip

We now consider each case in turn.

\begin{case}
We first consider the case in which both conditions
\eqref{eq:condition1} and \eqref{eq:condition2} hold, and so we may
apply Proposition \ref{prop:one quasiconvex}.  In this case,
consequences \eqref{eq:prop:one quasiconvex2} and \eqref{eq:prop:one
  quasiconvex3} imply
\[ d(\mathcal{D}, g w_n x_0) \ge R t + d(g x_0, g w_n x_0) - 2 \gp{g
  x_0}{x_0}{g w_n x_0} - 2 C - B,  \]
where $B$ and $C$ only depend on $\delta$ and $Q$.  As condition
\eqref{eq:condition2} holds, we are in the case in which $d(g x_0, g
w_n x_0) \ge 2R$, and as we have chosen $R \ge B + 2C$ this implies
that
\[ d(\mathcal{D}, g w_n x_0) \ge Rt + R - 2 \gp{g x_0}{x_0}{g w_n
  x_0}, \]
and so $\phi_R(g w_n x_0) \le t + 1 - r$ may only occur if $\gp{g
  x_0}{x_0}{g w_n x_0} \ge R r/ 2$. By exponential decay for shadows
this happens with probability at most $K c^{R r/2}$, which is at most
$q^{r}/3$, for $r \ge 1$, and at most $K \ge 1$ for $r=0$, by
\eqref{eq:backtrackR}, as we have chosen $R \ge 2 \log(q^2 / 3K) /
\log(c)$. Therefore this case contributes an amount $q^r/3$ to the
upper bound for the probability that $\phi_R(g w_n x_0) = t + 1 - r$,
for all $1 \le r \le t+1$, and an amount $1$ to the upper bound for
$\phi_R(g w_n x_0) = t+1$.
\end{case}

\begin{case}
We now consider the case in which condition
\eqref{eq:condition1} fails, i.e.
\begin{equation} \label{eq:backtrack2}
\gp{g x_0}{x_0}{g w_n x_0} \ge Rt - C - A. 
\end{equation}
By exponential decay for shadows, Lemma \ref{lemma:shadow decay}, the
probability that condition \eqref{eq:backtrack2} does not hold is at
most $K c^{R t - C - A}$, which is at most $q^{t+1}/3$, for $t \ge 1$,
as by \eqref{eq:backtrackR} we have chosen $R \ge \log(q^2 / 3 K ) /
\log (c) + A + C$. In this case there is no restriction on the
possible value of $\phi_R(g w_n x_0)$, and so this case contributes an
amount of $q^{t+1}/3$ to the upper bound for the probability that
$\phi_R(g w_n x_0) = r$ for every possible value of $r \in \N_0$.
\end{case}

\begin{case}
We now consider the final case in which the sample path travels
distance at most $2R$, i.e. condition \eqref{eq:condition2} fails. By
linear progress with exponential decay, Theorem \ref{theorem:linear
  progress},
\[ \P(d(g x_0 , g w_n x_0 ) \le L n) \le K c^n. \] 
By \eqref{eq:backtrackN} we have chosen $N \ge 2R / L$, and so this
implies that
\[ \P(d(g x_0 , g w_n x_0 ) \le 2R) \le K c^n, \] 
for all $n \ge N$. Also by \eqref{eq:backtrackN} we have chosen $N
\ge \log(q^3 / 3 K) / \log(c)$, and so $K c^n \le q^3
/ 3$ for all $n \ge N$. As the sample path has not gone very far, then
the distance from $\mathcal{D}$ to $g w_n x_0$ can not have decreased
too much. In fact, by the triangle inequality, as
\[ d(\mathcal{D}, g w_n x_0) \ge d(\mathcal{D}, g x_0) - d(g x_0, g
w_n x_0), \] 
this implies that
\[ d(\mathcal{D}, g w_n x_0) \ge R t - 2R,  \]
i.e. $\phi_R( g w_n x_0 ) \ge \phi_R(g x_0 ) - 2$ in this case, which
occurs with probability at most $q^3 / 3$. This case contributes an
amount of $q^3/3$ to the upper bound for those values of $r \in \N_0$
satisfying $r \ge t - 2$.
\end{case}

At least one of the three cases above must occur, and so the desired
upper bounds arise from summing the probabilities in each case, which
we summarize in the table below.

\setlength{\tabcolsep}{2pt}
\begin{center}
\begin{tabular}{rlrlccccc} 
& & & & \multicolumn{4}{c}{upper bound} & \\ 
\multicolumn{2}{c}{$r$} & \multicolumn{2}{l}{Value of $\phi_R(g w_n x_0) = t+1-r$} & Case 1       & Case 2   & Case 3 & Total & \\
$4$ & $\le r \le t+1$ & $0$   & $ \le t + 1 - r \le t-3$                 &
$q^{r}/3 $ & $ + \ q^{t+1}/3$    & & $\le q^{r}$ \\
$1$ & $\le r \le 3$ & $t-2$ & $ \le t + 1 - r \le t$                   &
$q^{r}/3 $ & $ + \ q^{t+1}/3 $  & $ + \ q^3/3$ & $\le q^{r}$ \\
& $r = 0$ & $t+1$ & $ \le t + 1 - r $                        & 1
& $ + \ q^{t+1}/3$   & $ + \ q^3/3$ & $\le 1$ \\ 
\end{tabular}
\end{center}
\setlength{\tabcolsep}{6pt}
The Total column gives an upper bound on probability that any of the
cases occur, which in the final row is the trivial upper bound of $1$.

This completes the proof of Proposition \ref{prop:backtrack}.
\end{proof}

\section{Train tracks and shadows} \label{section:train track}

In this section we collect some useful facts about train tracks and
shadow sets, as defined in \eqref{eq:shadow def}. The key observation
is that a maximal recurrent train track determines a subset of the
curve complex which contains a shadow set.

We briefly review some properties of train tracks, see Penner and
Harer \cite{penner-harer} for more details. A \emph{train track} is a
smoothly embedded $1$-complex $\tau$ on a surface such that edges
(called \emph{branches}) at each vertex (called a \emph{switch}) are
all tangent, and there is at least one edge in both possible tangent
directions at each vertex. Therefore, for each switch, the branches
are divided into two non-empty sets of branches with the same signed
tangent vector, which are called the \emph{incoming} and
\emph{outgoing} branches. The complementary regions are surfaces with
boundaries and cusps, and none of the complementary regions may be
annuli or discs with two or fewer cusps. A train track on a closed
surface is \emph{maximal} if every complementary region is a triangle.

A train route is a smooth path in $\tau$, and so it crosses a switch
by going from an incoming branch to an outgoing branch, or vice
versa. A train track is \emph{recurrent} if every branch is contained
in a closed train route.  A \emph{transverse measure} on $\tau$ is a
non-negative function on the branches which satisfies the \emph{switch
  condition}, i.e. at each switch, the sum of the measures of the
incoming branches is equal to the sum of the measures of the outgoing
branches. Any closed train route induces a transverse measure on
$\tau$ given by the counting measure. We shall write $P(\tau)$ for the
collection of transverse measures supported by $\tau$, which is a
subset of $\mathcal{ML}$, Thurston's space of measured laminations on
the surface, and $P(\tau)$ is a cone on a compact polyhedron in
$\mathcal{PML}$, the projectivization of $\mathcal{ML}$. If the train
track is maximal, then $P(\tau)$ has the same dimension as
$\mathcal{ML}$.

A simple closed curve $x$ is \emph{carried} on $\tau$ if $x$ is
homotopic to a train route. A train track $\sigma$ is \emph{carried}
by a train track $\tau$ if every train route on $\sigma$ is homotopic
to a train route on $\tau$, and we denote this by $\sigma \prec \tau$.

Given a train track $\tau$, we may produce new train tracks by
\emph{splitting} $\tau$, by the following local modification,
illustrated in Figure \ref{pic:splits}, in which a subset of the train
track corresponding to the top configuration is replaced by one of the
three lower configurations.

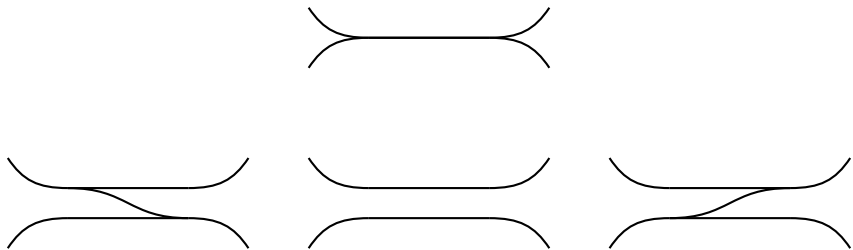
\begin{figure}[H] \begin{center}
\begin{tikzpicture}[scale=0.4]

\def\tangent{
\draw[thick] (-2, 1) .. controls (-1.5, 0.25) and (-1, 0) .. (0, 0);
}

\def\segment{
\tangent
\draw[thick] (0, 0) -- (4, 0);
\begin{scope}[xshift=4cm,x=-1cm]
        \tangent
\end{scope}
}

\segment
\begin{scope}[y=-1cm]
        \segment
\end{scope}

\def\split{
\begin{scope}[yshift=-5cm]
        \segment
\end{scope}

\begin{scope}[yshift=-6cm,y=-1cm]
        \segment
\end{scope}
}

\split

\begin{scope}[xshift=10cm]
        \split
        \draw[thick] (0, -6) .. controls (2, -6) and (2, -5) .. (4, -5);
\end{scope}

\begin{scope}[xshift=-10cm]
        \split
        \draw[thick] (0, -5) .. controls (2, -5) and (2, -6) .. (4, -6);
\end{scope}

\end{tikzpicture}
\end{center} 
\caption{Splitting a train track.}
\label{pic:splits}
\end{figure}

The central configuration is called a \emph{collision} or
\emph{degenerate split}. For our purposes, we will only need to
consider non-degenerate splits.

Given a recurrent maximal train track $\tau$, let $C(\tau)$ be the
collection of simple closed curves carried by $\tau$. 

We start by showing that if two shadow sets intersect, then we may
increase the parameter of one of them by a bounded amount such that
the new shadow set contains both of the original ones.

\begin{prop} \label{prop:intersect} %
There is a number $R$, which only depends on $\delta$, such that if
$S_{x_0}(x_1, R_1)$ and $S_{x_0}(x_2, R_2)$ intersect, then
$S_{x_0}(x_1, R_1) \subset S_{x_0}(x_2, \min \{ d(x_0, x_2) - d(x_0,
x_1) + R_1, R_2 \} + R)$.
\end{prop}

\begin{proof}
Let $y$ be a point in the intersection $S_{x_0}(x_1, R_1) \cap
S_{x_0}(x_2, R_2)$. As $y$ lies in both shadow sets, $\gp{x_0}{x_1}{y}
\ge d(x_1, y) - R_1$ and $\gp{x_0}{x_2}{y} \ge d(x_2, y) -
R_2$. Recall that for any points $x, y$ and $z$, we have
\begin{equation} \label{eq:gp} %
\gp{x_0}{x}{z} \ge \min \{ \gp{x_0}{x}{y}, \gp{x_0}{y}{z} \} - \delta,
\end{equation} 
see for example Bridson and Haefliger \cite{bh}*{III.H.1.20}. Applying
\eqref{eq:gp} to the three points $x_1, x_2$ and $y$ implies that
$\gp{x_0}{x_1}{x_2} \ge \min \{ d(x_0, x_1) - R_1, d(x_0, x_2) - R_2
\} - \delta$. Now let $z$ be a point in $S_{x_0}(x_1, R_1)$, so
$\gp{x_0}{x_1}{z} \ge d(x_0, z) - R_1$. Similarly, applying
\eqref{eq:gp} to the three points $x_1, x_2$ and $z$ implies that
$\gp{x_0}{x_2}{z} \ge \min \{ d(x_0, x_1) - R_1, d(x_0, x_2) - R_2 \}
- 2 \delta$, so we may choose $R = 2 \delta$.
\end{proof}

A shadow set is always non-empty as long as $R \ge 0$, and
furthermore, if the parameter $R$ is sufficiently large, the shadow
sets have non-empty limit sets in $\partial \mathcal{C}(S_g)$. We say
a group $G$ acts coarsely transitively on $X$ if there is a number $K$
such that for any $x$ and $y$ in $X$, there is a group element $g$
such that $d(gx, y) \le K$. The action of $\MCG(S_g)$ on
$\mathcal{C}(S_g)$ is coarsely transitive.

\begin{prop} \cite{bhm} %
Let $X$ be a Gromov hyperbolic space, which need not be locally
compact, but whose isometry group acts coarsely transitively on $X$.
There is a number $R_0 \ge 0$, which only depends on $\delta$, such
that for all $R \ge R_0$ the shadow set $S_{x_0}(x, R)$ has a limit
set in $\partial X$ which contains a non-empty open set, for all $x$
and $x_0$.
\end{prop}

For any number $A > 0$, the shadow set $S_{x_0}(x, R)$ is contained in
$S_{x_0}(x, R + A)$, and furthermore, if $A$ is sufficiently large,
then the limit sets of the shadow sets are strictly nested, and the
distance between them in $\mathcal{C}(S_g)$ is bounded below in terms of
$A$.

\begin{lemma} \cite{M2} \label{lemma:nested shadows} %
Let $X$ be a $\delta$-hyperbolic space, which need not be locally
compact.  There is a number $K$, which depends only on $\delta$,
such that for all positive numbers $A$ and $R$, and any $x, y \in X$
with $d(x, y) \ge A + R + 2K$, the closure of the shadow $S_x(y, R)$
is disjoint from the closure of the complement of the shadow $S_x(y, R
+ A + K)$, i.e.
\[ \overline{ S_x(y, R) } \cap \overline{ X \setminus S_x(y, R + A +
  K) } = \varnothing. \]
Furthermore for any pair of points $a, b \in X$ such that $a \in
S_x(y, R)$ and $b \in X \setminus S_x(y, R + A + K)$, the distance
between $a$ and $b$ is at least $A$.
\end{lemma}

We will use these two properties in the form of the following
elementary corollary, which says that every shadow set contains shadow
sets nested inside it by an arbitrarily large distance.

\begin{cor}
There is a number $R_0 \ge 0$ such that for any number $A \ge 0$,
and for any shadow set $S_{x}(y, R_0)$, there is a shadow set
$S_{x}(z, R_0) \subset S_{x}(y, R_0)$ with $\overline{\mathcal{C}(S_g)
  \setminus S_{x}(y, R_0)} \cap \overline{S_{x}(z, R_0)} =
\varnothing$ and $d(\mathcal{C}(S_g) \setminus S_{x}(y, R_0), S_{x}(z,
R_0) ) \ge A$.
\end{cor}

We also observe that the complement of a shadow set is roughly a
shadow set.

\begin{lemma} \cite{M2} \label{lemma:shadow complement} %
There is a number $K$, which only depends on $\delta$, such that for
all numbers $R \ge 2K$, and all $x, z \in \mathcal{C}(S_g)$ with $d(x,
z) \ge R + K$,
\[ S_x(z, d(x,z) - R - K ) \subset \mathcal{C}(S_g) \setminus S_z(x, R)
\subset S_x(z, d(x,z) - R + K).  \]
\end{lemma}

Finally, we recall the following ``change of basepoint'' result for
shadow sets.

\begin{lemma} \cite{M2} \label{lemma:shadow basepoint} %
There are numbers $A$ and $B$, which only depend on $\delta$, such
that for any $r$, and any three points $x, y$ and $z$ with
$\gp{z}{x}{y} \le r - A$, there is an inclusion of shadows, $S_z (x, r)
\subset S_y (x, s)$, where $s = d(x, y) - d(x, z) + r - B$.
\end{lemma}

We now provide a link between subsets of the curve complex determined
by train tracks and shadow sets, by showing that every maximal train
track $\tau$ determines a subset $C(\tau)$ of the curve complex which
contains a shadow set.

We shall write $\mathcal{L}_{min}(S_g)$ for the set of all laminations
corresponding to minimal foliations, i.e. those laminations which
neither contain simple closed curves, and are not disjoint from any
simple closed curves. This is a subset of $\vPML(S_g)$ with the
relative topology. The set of ending laminations $\mathcal{EL}(S_g)$
is a quotient of $\mathcal{L}_{min}(S_g)$ by the equivalence relation
of topological equivalence, i.e. two measured laminations are
identified if they correspond to different measures on the same
topological lamination. We shall write $\mathcal{T}(S_g)$ for the
Teichm\"uller space of the surface $S_g$, which is the space of
hyperbolic metrics on the surface. We shall write $\iota$ for the
coarsely well-defined map $\iota \colon \mathcal{T}(S_g) \to
\mathcal{C}(S_g)$, which sends a point in Teichm\"uller space to a
simple closed curve on the surface of shortest length with respect to
the corresponding metric. Klarreich \cite{klarreich}, see also
Hamenst\"adt \cite{hamenstadt}, showed the Gromov boundary of the
complex of curves is homeomorphic to the space of ending laminations.

\begin{theorem} \cite{klarreich} %
The inclusion map $\iota \colon \mathcal{T}(S_g) \to
\mathcal{C}(S_g)$ extends continuously to the portion
$\mathcal{L}_{min}(S_g)$ of minimal laminations of $\vPML(S_g)$ to
give a map $\pi \colon \mathcal{L}_{min}(S_g) \to \partial
\mathcal{C}(S_g)$. The map $\pi$ is surjective and $\pi(\mathcal{L}) =
\pi(\mathcal{L}')$ if and only if $\mathcal{L}$ are topologically
equivalent, and in fact $\pi$ induces a homeomorphism between
$\mathcal{EL}(S_g)$ and $\partial \mathcal{C}(S_g)$.
\end{theorem}

We now use this to show that every maximal recurrent train track
$\tau$ determines a subset $C(\tau)$ of the complex of curves which
contains a shadow set.

\begin{prop} \label{prop:maximal} %
For any simple closed curve $x$ and any maximal recurrent train track
$\tau$ there is a simple closed curve $y$ and a number $R \ge R_0$
such that the shadow set $S_{x}(y, R)$ is contained in $ C(\tau)$.
\end{prop}

\begin{proof}
Let $S_{x}(y, R)$ be any shadow set with $R \ge R_0$, with $d(x, y)$
sufficiently large such that the closure of $S_x(y, R)$ is not equal
to all of $\mathcal{C}(S_g)$.

Given a pseudo-Anosov element $\phi$ in $\MCG(S_g)$, we shall write
$(\mathcal{L}^s(\phi), \mathcal{L}^u(\phi))$ for the pair in
$\mathcal{PML} \times \mathcal{PML}$ consisting of its stable and
unstable laminations. Such pairs $(\mathcal{L}^s(\phi),
\mathcal{L}^u(\phi))$ are dense in $\mathcal{PML} \times
\mathcal{PML}$, as $\phi$ runs over all pseudo-Anosov elements in
$\MCG(S_g)$, see for example \cite{M1}*{Lemma 3.4}. Furthermore, a
pseudo-Anosov element $\phi$ acts with \emph{north-south dynamics} on
the Thurston compactification $\mathcal{T}(S_g) \cup \vPML(S_g)$, with
fixed points the stable and unstable laminations $\mathcal{L}^s(\phi)$
and $\mathcal{L}^u(\phi))$ of $\phi$. This means that for any open set
$U$ containing $\mathcal{L}^s(\phi)$ and disjoint from
$\mathcal{L}^u(\phi)$, and for any closed set $V$ disjoint from the
fixed points, there is a number $N$, depending on $U, V$ and $g$ such
that $g^n(V) \subset U$ for all $n \ge N$.

Therefore, it suffices to show that there is a shadow set $S_x(y, R)$,
such that the closure of $\iota^{-1}(S_x(y, R))$ in the Thurston
compactification $\mathcal{T}(S_g) \cup \vPML(S_g)$ is disjoint from
an open set in $\vPML$. Suppose not, then the closure of
$\iota^{-1}(S_x(y, R))$ is dense in $\vPML$. As the ending laminations
$\mathcal{EL}$ are dense in $\vPML$, and their image under $\iota$ is
equal to $\partial \mathcal{C}(S_g)$, this implies that the closure of
the shadow $S_x(y, R)$ is equal to all of $\partial \mathcal{C}(S_g)$,
a contradiction.
\end{proof}

We now show that if the closure of a shadow set $S_x(y, R)$ contains a
limit point of the subgroup supporting the random walks, then there is
a slightly larger shadow set $S_x(y, R + R_1)$ whose closure has
positive measure with respect to the hitting measure $\nu$.

\begin{prop} \label{prop:positive} %
There is a number $R_1 > 0$ such that if $S_{x_0}(y, R)$ contains a
limit point of $H$, then $\nu(\overline{S_{x_0}(y, R + R_1)} ) >
0$. Furthermore, for any number $D$, there is a shadow set $S_{x_0}(z,
R_0) \subset S_{x_0}(y, R + R_1)$, distance at least $D$ from $x_0$,
with $\nu( \overline{S_{x_0}(z, R_0)} ) > 0$.
\end{prop}

\begin{proof}
There is a sequence of group elements $( h_i )_{i \in \N}$ such that
$( h_i x_0 )_{i \in \N}$ converges to $\mathcal{L}$ in $S_{x_0}(y,
R)$, so infinitely many of the $( h_i x_0 )_{i \in \N}$ lie in
$S_{x_0}(y, R + R_1)$. Choose one with $\mu_n(h_i) > 0$ for some
$n$. Then by exponential decay for shadows, a definite proportion of
sample paths starting from $h_i x_0$ at time $n$ converge into
$S_{x_0}(y, R + R_2)$, so $\nu( \overline{S_{x_0}(y, R + R_2)} ) > 0$,
as required.

For any $D$, the countable collection of sets $S_{x_0}(z, R_0)$, as
$z$ runs over all vertices of the curve complex $\mathcal{C}(S_g)$
with $d(x_0, z) \ge D$, cover $S_{x_0}(y, R + R_2)$, so at least one of
these has positive measure, and is contained in $S_{x_0}(y, R + R_3)$,
for some $R_3$.
\end{proof}

\section{Local estimates close to the disc set} \label{section:local near}

The main purpose of this section is to show

\begin{prop} \label{prop:local} %
For any complete subgroup $G < \MCG(S_g)$ there is a finitely
generated non-elementary subgroup $H < G$, such that any finitely
supported probability distribution $\mu$, whose semi-group support
$\langle \supp(\mu) \rangle_+$ is a subgroup containing $H$, has the
following property.

For any number $R$ there are numbers $N$ and $\e > 0$, depending only
on $R$ and $\mu$, such that for any mapping class group element $g$,
there is a definite probability $\e$ that the random walk of length
$n$ generated by $\mu$ starting at $g$ is distance at least $R$ from
the disc set $\mathcal{D}$, so in particular
\[ \P( \phi_R(g w_{n} x_0) = 1 \mid \phi_R( g x_0) = 0) \ge \e, \]
for all $n \ge N$
\end{prop}

We will construct such a subgroup $H$ using the following result of
Kerckhoff \cite{Ker}.

\begin{theorem} \cite{Ker}*{Proposition on p36}
\label{theorem:kerckhoff}
There is a recurrent maximal train track $\tau$ on a closed orientable
surface $S_g$, such that for any identification of $S_g$ with the
boundary of a handlebody, $\tau$ can be split at most $-9\chi(S_g)$
times to a recurrent maximal train track $\tau'$ such that $N(\tau')$
is disjoint from the disk set $\mathcal{D}$ of the handlebody. Here
$\chi(S_g)$ is the Euler characteristic of the surface.
\end{theorem}

We now prove Proposition \ref{prop:local}.

\begin{proof}(of Proposition \ref{prop:local}.)  %
Let $\tau$ be a recurrent maximal train track, satisfying Theorem
\ref{theorem:kerckhoff}, and let $T$ be the finite collection of
maximal train tracks obtained by splitting $\tau$ at most $-9
\chi(S_g)$ times.

By Proposition \ref{prop:maximal}, the subset of the curve complex
$C(\tau_i)$ corresponding to each maximal train track $\tau_i \in T$
contains a shadow set which we shall denote $S_i = S_{x_0}(y_i,
R_0)$. Furthermore, for any number $R \ge 0$, for each $S_i$ we may
choose a nested shadow set $S'_i = S_{x_0}(y'_i, R_0)$, with
$\overline{S'_i} \cap \overline{S_i \setminus \mathcal{C}(S_g)} =
\varnothing$ and $d(\mathcal{C}(S_g) \setminus S_i, S'_i) \ge R$, for
each $i$.

The subgroup $G$ is complete in $\MCG(S_g)$, and so endpoints of
pseudo-Anosov elements are dense in $\partial \mathcal{C}(S_g)$. In
fact, the pairs $(F_+, F_-)$ of stable and unstable laminations are
dense in $\partial \mathcal{C}(S_g) \times \partial \mathcal{C}(S_g)$.
Each shadow set $S'_i$ contains a non-empty open set in $\partial
\mathcal{C}(S_g)$, so for each $S'_i$, choose a pseudo-Anosov element
$g_i$, at least one of whose limit points lies in $S'_i$. Let $H$ be
the finitely generated subgroup generated by the finite list of
elements $g_i$.

By Proposition \ref{prop:positive}, for any finitely supported
probability distribution $\mu$ whose semi-group support $\langle
\supp(\mu) \rangle_+$ is a group containing $H$, the hitting measure
$\nu( \overline{S_i} ) > 0$. Set $\e = \min \nu( \overline{S_i}
)/2$. As the convolution measures $\mu_n$ weakly converge to $\nu$,
there is an $N$ such that $\mu_n( \overline{S_i} ) \ge \e$ for all $n
\ge N$ and for all $i$.

Now consider a random walk of length $n$ starting from $g x_0$. We
wish to estimate the distance $d(\mathcal{D}, g w_n x_0)$, and by
applying the isometry $g^{-1}$, this is equivalent to considering the
distance $d(g^{-1}\mathcal{D}, w_n x_0)$. By Theorem
\ref{theorem:kerckhoff}, for any disc set, in particular the disc set
$g^{-1} \mathcal{D}$, there is some maximal train track $\tau_i$ in
$T$ disjoint from $g^{-1} \mathcal{D}$, and so if $w_n x_0$ lies in
the nested shadow set $S'_i$ contained in $C(\tau_i)$, then $d(g^{-1}
\mathcal{D}, w_n x_0) \ge R$, and so $\phi_R(g w_n x_0 ) \ge 1$. As
$\mu_n(S'_i) \ge \e$ for all $n \ge N$, this implies that $w_n x_0$
lies in $S'_i$ with probability at least $\e$, and so this completes
the proof of Proposition \ref{prop:local}.
\end{proof}

\section{Exponential decay for distance from the disc set} \label{section:distance from D}

We now use the local estimates for distance from the disc set obtained
in the previous sections to show that the distance from $w_n x_0$ to
the disc set grows linearly with exponential decay.

\begin{prop} \label{prop:distance from D} %
For any complete subgroup $G$ of $\MCG(S_g)$, there is a finitely
generated subgroup $H < G$, such that for any finitely supported
probability distribution $\mu$ whose semi-group support $\langle
\supp(\mu) \rangle_+$ is a subgroup containing $H$, there are numbers
$K, L > 0$ and $c < 1$, which depend on $\mu$, such that for any disc
set $\mathcal{D}$ and any basepoint $x_0$, the distance of $w_n x_0$
from the disc set $\mathcal{D}$ grows linearly with exponential decay,
i.e.
\[ \P( d(\mathcal{D}, w_n x_0) \le Ln   ) \le K c^n.  \]
\end{prop}

We show this by comparing the distribution of the random variables
$d(\mathcal{D}, w_n x_0)$ with a Markov chain on $\N_0$ which gives an
upper bound on the probability that the random variable takes small
values. Ultimately, we show the Markov chain $(\N, P_0)$ has spectral
radius $\rho(P)$ strictly less than $1$, and this gives the
exponential decay result we require.

One minor technical point is that the local estimates hold for all $n
\ge N$. We now observe that if an exponential decay estimate holds for
the iterated random walk $w_{nN}$ generated by the $N$-fold
convolution measure $\mu_{N}$, then a similar exponential decay
estimate holds for $w_n$, though for different constants.

\begin{prop}
Let $\mu$ be a finitely supported probability distribution such that
for some $N > 0$, there are numbers $K, L > 0$ and $c < 1$ such that
\[ \P( d(\mathcal{D}, w_{nN} x_0) \le Ln   ) \le K c^n.  \]
Then there are numbers $L' > 0$ and $c' < 1$ such that
\[ \P( d(\mathcal{D}, w_{n} x_0) \le L'n   ) \le K {c'}^n,  \]
for all $n$.
\end{prop}

\begin{proof}
As $\mu$ has finite support, with diameter $R$ say, $d(\mathcal{D},
w_{nN + k} x_0 ) \le d(\mathcal{D}, w_{nN} x_0) + RN$, for any $ 0 \le
k \le N$, so we may choose $L' = L + RN$ and $c' = c^{1/N}$.
\end{proof}

This shows that by replacing $\mu$ with $\mu_{N}$, we may assume that
the local estimates hold for $N = 1$, and we shall do this for the
remainder of this section to simplify notation. We now consider the
sequence of random variables $X_n \colon \Omega \to \N_0$ defined by
$X_n(\omega) = \phi_R(w_{n} x_0 )$.

We wish to compare the distributions of the $X_n$ with the
distributions $X'_n$ arising from the following Markov chain $(\N_0,
P)$, starting with total mass $1$ at $0 \in \N_0$ at time $n = 0$.  We
shall write $p(i, j)$ for the probability you go from $i$ to $j$.
\[ p(0, j) = \left\{ \begin{array}{cl} 1 - \e & \text{ if } j = 0 \\
\e & \text{ if } j = 1 \\ 0 & \text{ if } j \ge 2 \\ \end{array}
\right. \]
and for $i > 0$,
\[ p(i, j) = \left\{ \begin{array}{cl} q^{i - j + 1} & \text{ if } j
\le i \\ 1 - q - q^2 - \cdots q^{i+1} & \text{ if } j = i + 1 \\ 0 &
\text{ if } j \ge i + 2 \end{array} \right. \]
Figure \ref{pic:markov-chain} illustrates the first few vertices of
this Markov chain.

Given a probability measure $P$ on $\N_0$, we shall write $F_P$ for
the cumulative distribution function of $X$, i.e.
\[ F_P(n) = \sum_{i = 0}^n P( i ). \]
Given two probability measures $P$ and $P'$ on $\N_0$, we say that
$P'$ \emph{stochastically dominates} $P$ if $F_P(n) \ge F_{P'}(n)$ for
all $n$, and we shall denote this by $P \dom P'$.  Similarly, given a
random variable $X$ which takes values in $\N_0$, i.e. $X \colon (I,
\P) \to \N_0$, we shall write $F_X$ for the cumulative distribution
function of $X$. We say a random variable $X$ \emph{stochastically
  dominates} a random variable $X'$ if $F_{X}(n) \le F_{X'}(n)$ for
all $n$, and we shall denote this by $X' \dom X$.

Let $\{ X_n \}$ be a sequence of random variables with values in
$\N_0$. The \emph{transition kernels} for the sequence are the
measures on $\N_0$ given by
\[ K_n(i_1, \ldots, i_{n-1})(A) = \P(X_n \in A \mid (X_1, \ldots,
X_{n-1}) = (i_1, \ldots i_{n-1}) ), \]
where $A \subset \N_0$.  If the transition kernels for two sequences of
random variables $\{ X_n \}$ and $\{ X'_n \}$ satisfy $K'_n(i_1,
\ldots i_{n-1}) \dom K_n(i_1, \ldots i_{n-1})$ for all $n$ and $i_1,
\ldots i_{n-1}$, then in fact $X'_n \dom X_n$ for all $n$. We now
state a precise form of this result, which is often referred to as
Strassen's Theorem, or the Stochastic Domination Theorem, see Lindvall
\cite{lindvall}*{Chapter IV}. This result holds in much greater
generality, and the version we state here is a simplified one
sufficient for our purposes.

\begin{theorem} \cite{lindvall}*{Theorem 5.8} %
Let $\{ X_n \}$ and $\{ X'_n \}$ be sequences of random variables with
values in $\N_0$ such that $X'_0 \dom X_0$, and $K'_n(i_1, \ldots
i_{n-1}) \dom K_n(i_1, \ldots i_{n-1})$, for all $n$ and $i_1, \ldots
i_{n-1}$. Then
\[ X'_n \dom X_n,   \]
for all $n$.
\end{theorem}

As in our case $X'_0 = X_0$, in order to show that $X'_n \dom X_n$, it
suffices to show that the transition kernels satisfy $K'_n \dom K_n$,
and we show this using the local estimates, Propositions
\ref{prop:backtrack} and \ref{prop:local}.

\begin{prop}
For all $n$, and all $i_1, \ldots, i_{n-1}$, the transition kernels
for $\{ X_n \}$ and $\{ X'_n \}$ satisfy
\[ K'_n(i_1, \ldots, i_{n-1}) \dom K_n(i_1, \ldots, i_{n-1}). \]
\end{prop}

\begin{proof}
As the random variables $\{ X'_n \}$ arise from a Markov chain,
$K'_n(i_1, \ldots, i_{n-1})$ only depends on the value of $i_{n-1}$,
and may be computed from the defining transition probabilities $p(i,
  j)$ of the Markov chain.

We first consider the case in which $i_{n-1} = 0$.  By Proposition
\ref{prop:local},
\begin{equation} \label{eq:epsilon} %
\P( X_{n}(\omega) = 0 \mid X_{n-1}(\omega) = 0 ) \le 1 - \e,
\end{equation}
for all $n$, and writing this out in terms of the cumulative
probability distributions gives
\[ F_{K_n(i_1, \ldots, i_{n-2}, 0)}(0) \le F_{K'_n(i_1, \ldots
  i_{n-2}, 0)}(0).  \]
By definition of the Markov chain, the value of $X'_n$ may not
increase by more than one from the value of $X'_{n-1}$, and so
$F_{K'_n(i_1, \ldots i_{n-2}, 0)}(1) = 1$, and so this shows that
\[ K'_n(i_1, \ldots   i_{n-2}, 0) \dom K_n(i_1, \ldots
  i_{n-2}, 0) \]
for all $n$ and all $i_1, \ldots, i_{n-2}$.

We now consider the case in which $i_{n-1} > 0$.  By Proposition
\ref{prop:backtrack},
\begin{equation} \label{eq:q} %
\P( X_{n}(\omega) = k-i \mid X_{n-1}(\omega) = k ) \le q^{i+1}, 
\end{equation}
for all $0 \le i \le k$, and for all $n$, and so by summing over
values of $i$ with $l \le i \le k$, this implies that
\[ F_{K_n(i_1, \ldots i_{n-2}, k)}(l) \le F_{K'_n(i_1, \ldots
  i_{n-2}, k)}(l), \]
for all $l \le k$.  Again, by definition of the Markov chain, the
value of $X'_n$ may not increase by more than $1$ at any step, and so
$F_{K'_n(i_1, \ldots i_{n-2}, k)}(k+1) = 1$, and so this shows that
\[ K'_n(i_1, \ldots   i_{n-2}, k) \dom K_n(i_1, \ldots
  i_{n-2}, k), \]
for all $n$ and all $i_1, \ldots i_{n-2}$, and with $k > 0$.
\end{proof}

We now show that the Markov chain $(\N_0, P)$ has spectral radius
$\rho(P) < 1$. If $f$ is a function on $\N_0$ we shall write $P f$ to
denote the function $P f(k) = \sum p(k, j) f(j)$.

\begin{prop}
If $q <  1/4$ then the Markov chain $(\N_0, P)$ has spectral radius
$\rho < 1$.
\end{prop}

\begin{proof}
Recall from Woess \cite{woess}*{Section 7} that $\rho \le t$ if there
is a strictly positive $t$-superharmonic function $f$ on $\N_0$, i.e
the function $f$ satisfies $f(k) > 0$ for each $k$, and $P f \le t
f$. We will show that if we choose $t = \max \{ 1 - \e(1-2q), 4q \}$, then the function
$f(k) = (2q)^k$ is $t$-superharmonic for $t < 1$ as long as $q <  1/4$. This may
be verified by an elementary calculation, and we provide the details
below for the convenience of the reader.

The first inequality $P f(k) \le t f(k)$, for $k=0$, is
\begin{equation} \label{eq:harmonic1} 
(1 - \e)f(0) + \e f(1) \le t f(0), 
\end{equation}
and then the remaining inequalities are of the form
\begin{equation} \label{eq:harmonic3} 
q^{k+1} f(0) + q^k f(1) + \cdots + q f(k) + p_k f(k+1) \le t
f(k), 
\end{equation}
for $k \ge 1$, where $p_k = 1 - q - q^2 - \cdots q^{k+1}$.

The first inequality \eqref{eq:harmonic1}, gives
\[ 1 - \e(1 - 2q) \le t,  \]
which is satisfied for some $t < 1$ if $q < 1/2$.

For the general case \eqref{eq:harmonic3}, we obtain
\[ q^{k+1} + q^k(2q) + \cdots + q(2q)^k + p_k (2q)^{k+1}  \le t (2q)^k.  \]
As $p_k \le 1$, this inequality is satisfied if
\[ q^{k+1}(1 + 2 + \cdots 2^k) + (2q)^{k+1} \le t (2q)^k.  \]
As the sum of the geometric series is less than $2^{k+1}$, it suffices
to choose $q$ such that
\[ 4q  \le t,  \]
and this holds for some $t < 1$ if $q < 1/4$, as required.
\end{proof}

We shall write $p^{(n)}(i, j)$ for the probability that the Markov
chain starting at $i$ at time $0$ is at location $j$ on the $n$th
step. A Markov chain on a graph is \emph{uniformly irreducible} if
there are numbers $N$ and $\e_0$ such that for any pair of
neighbouring vertices $i$ and $j$, $p^{(n)}(x, y) \ge \e_0$ for some
$n \le N$.  If we consider $\N_0$ to have the standard graph structure
in which $i$ and $j$ are connected by an edge if and only if
$\norm{i-j} \le 1$, then the Markov chain $(\N_0, P)$ above is
uniformly irreducible, with $N=1$ and $\e_0 = \min \{ \e, q^2, 1 -
q/(1-q)\}$.

\begin{lemma} \cite{woess}*{Lemma 8.1} \label{lemma:irreducible} %
If the Markov chain $(X, P)$ is uniformly irreducible, then there is a
number $A > 0$ such that $p^{(n)}(x, y) \le A^{d(x, y)}\rho(P)^n$.
\end{lemma}

Therefore, by Lemma \ref{lemma:irreducible}, there is a number $A >
0$ such that $F_{X_n}(Ln) \le Ln A^{Ln} \rho^n$, and this decays
exponentially for some $L > 0$, chosen sufficiently close to zero.

\section{Exponential decay for Heegaard splitting
  distance} \label{section:splitting distance}

In this section we prove linear progress with exponential decay for
Heegaard splitting distance, Theorem \ref{Hg}, using linear progress
with exponential decay for distance from the disc set, Proposition
\ref{prop:distance from D}, and Proposition \ref{prop:two
  quasiconvex}, to estimate the distance between two quasiconvex sets.

\begin{proof} (of Theorem \ref{Hg}.)
It will be convenient to think of the sample path $w_n$ as the
concatenation of two sample paths of lengths roughly $n/2$. To be
precise, choose $m = \lfloor n/2 \rfloor$. We may then consider the
sample path $w_n$ to consist of two segments, an initial segment
$w_{m}$ of length $m$, and a final segment $w_{m}^{-1}w_n$, of length
$n - m$. These two sample paths $w_{m}$ and $w_{m}^{-1}w_n$ are
independently distributed.

The distance from $\mathcal{D}$ to $w_{m} x_0$ decays exponentially in
$n$, by Proposition \ref{prop:distance from D}, i.e.
\begin{equation} \label{eq:overlap1}
\P( d(\mathcal{D}, w_{m} x_0)  \le Ln/2 ) \le K c^{n/2}. 
\end{equation}
Proposition \ref{prop:distance from D} also applies to the reflected
random walk generated by the probability distribution $\rmu(g) =
\mu(g^{-1})$, with the same basepoint $x_0$, but with the disc set
$h_{S^3} \mathcal{D}$, though possibly with different numbers
$\check K, \check L > 0$ and $\check c < 1$. To simplify notation, we
shall write $\mathcal{D}'$ for $h_{S^3} \mathcal{D}$. Therefore
applying Proposition \ref{prop:distance from D} to the reflected
random walk of length $n - m$ gives
\[ \P( d( \mathcal{D}', (w_{m}^{-1} w_n)^{-1} x_0) \le \check L
(n/2+1) ) \le \check K {\check c \ }^{n/2+1}. \]
Now applying the isometry $w_n$ we obtain
\begin{equation} \label{eq:overlap2} %
\P( d(w_{m} x_0, w_n \mathcal{D}' ) \le \check L (n/2+1) ) \le \check
K {\check c \ }^{n/2+1}.
\end{equation}

Choose $L_1 = \min \{ L, \check L\}$, and let $y$ be a closest point
in $\mathcal{D}$ to $w_m x_0$, and similarly, let $y'$ be a closest
point in $\mathcal{D}'$ to $w_m x_0$. We may apply Proposition
\ref{prop:two quasiconvex}, with the quasiconvex sets chosen to be the
two discs sets, and $y$ chosen to be $w_m x_0$, unless condition
\eqref{eq:prop:two quasiconvex1} fails, which in this context is
\begin{equation} \label{eq:sd} %
\gp{w_m x_0}{y}{y'} \le L_1 (n/2 + 1) - A, 
\end{equation}
where $A$ is the constant from Proposition \ref{prop:two quasiconvex}.
However, we now show that the probability that this condition fails
decays exponentially in $n$. If we choose $N = 4A / L_1$, then $L_1
(n/2 + 1) - A \ge L_1 n /4 - C - B/2$, for all $n \ge N$, where $A, B$
and $C$ are the constants from Proposition \ref{prop:two
  quasiconvex}. Using the Gromov product estimate \eqref{eq:prop:two
  quasiconvex3} from Proposition \ref{prop:two quasiconvex}, and
exponential decay for shadows, there are numbers $K_1$ and $c_1$ such
that
\begin{equation} \label{eq:overlap3} %
\P \left( \ \gp{w_m x_0}{x_0}{w_n x_0} \ge L_1 n/4 - C - B/2 \ \right)
\le K_1 c_1^{L_1 n/4 - C - B/2}.
\end{equation}

Assuming \eqref{eq:sd} fails, Proposition \ref{prop:two quasiconvex}
line \eqref{eq:prop:two quasiconvex2} implies that the distance from
$\mathcal{D}$ to $w_n \mathcal{D}'$ is at least
\[ d(\mathcal{D}, w_n \mathcal{D}') \ge d(\mathcal{D}, w_{m} x_0) + d(
w_{m} x_0, w_n \mathcal{D}') - 2 \gp{w_{m} x_0}{y}{y'} - B, \]
where $B$ is a constant which only depends on $\delta$ and the
quasiconvexity constant $Q$.  By Proposition \ref{prop:two
  quasiconvex} line \eqref{eq:prop:two quasiconvex3}, the difference
between the Gromov products $\gp{w_m x_0}{y}{y'}$ and $\gp{w_m
  x_0}{x_0}{w_n x_0}$ is bounded, as $d(w_n x_0, w_n \mathcal{D}') =
d(x_0, \mathcal{D'})$, and we have chosen a basepoint $x_0$ which lies
in both $\mathcal{D}$ and $\mathcal{D}'$. This implies
\[ d(\mathcal{D}, w_n \mathcal{D}') \ge d(\mathcal{D}, w_{m} x_0) + d(
w_{m} x_0, w_n \mathcal{D}') - 2 \gp{w_{m} x_0}{x_0}{w_n x_0} - 2C - B, \]
where $C$ is a number which only depends on $\delta$ and $Q$.  The
three events whose probabilities are estimated in lines
\eqref{eq:overlap1}, \eqref{eq:overlap2} and \eqref{eq:overlap3} need
not be independent, but the probability that at least one of them
occurs is at most the sum of the probabilities that each
occurs. Therefore
\[ \P( d(\mathcal{D}, w_n \mathcal{D}') \le L_1 n/4 ) \le K c^{n/2} +
\check K {\check c \ }^{n/2 + 1} + K_1 c_1^{L_1 n/4 - C - B/2}, \]
for all $n \ge N$, which decays exponentially in $n$, as required.
\end{proof}



\vskip 20pt
\noindent Alexander Lubotzky \\
Hebrew University \\ 
\url{alex.lubotzky@mail.huji.ac.il} \\

\noindent Joseph Maher \\
CUNY College of Staten Island and CUNY Graduate Center \\
\url{joseph.maher@csi.cuny.edu} \\

\noindent Conan Wu \\
Princeton University \\
\url{conan777@gmail.com}


\end{document}